\documentclass[A4,12pt,reqno]{amsart}
\usepackage{amsmath,amssymb,amsthm,mathrsfs,amsopn,amsfonts,amsbsy,amscd,bm}
\usepackage{bbm,pifont,color,graphicx,latexsym,fancyhdr,CJK}
\usepackage{multirow,bigdelim,layout,enumerate,tikz}
\usepackage[all]{xy}
\usepackage[flushmargin,hang]{footmisc}
\usepackage[top=1in, bottom=1in, left=.618in, right=.618in]{geometry}

\newtheorem{thm}{Theorem}[section]
\newtheorem{defn}[thm]{Definition}
\newtheorem{lem}[thm]{Lemma}
\newtheorem{prop}[thm]{Proposition}
\newtheorem{cor}[thm]{Corollary}
\numberwithin{equation}{thm}
\theoremstyle{definition}
\newtheorem{rmk}[thm]{Remark}

\newtheorem{exam}[thm]{Example}

\newcommand{\opname}[1]{\operatorname{\mathsf{#1}}}

    \newcommand{\Ext}{\opname{Ext}}

\renewcommand{\mod}{\opname{mod}}

\newcommand{\Hom}{\opname{Hom}}
  \newcommand{\End}{\opname{End}}
  \newcommand{\soc}{\opname{soc}}

\newcommand{\Rep}{\opname{Rep}}

\newcommand{\diag}{\opname{diag}}

\newcommand{\row}{\opname{row}}   \newcommand{\col}{\opname{col}}
\newcommand{\dg}{\opname{dg}}

\newcommand{\az}{{\alpha}}

  \newcommand{\ttz}{{\Theta}}
\newcommand{\dz}{{\delta}}  \newcommand{\ddz}{{\Delta}}
\newcommand{\gz}{{\gamma}}  
\newcommand{\sz}{{\sigma}}  \newcommand{\ssz}{{\Sigma}}
  
\newcommand{\lz}{{\lambda}} \newcommand{\llz}{{\Lambda}}

\newcommand{\vz}{{\varphi}}

\newcommand{\ch}{{\mathcal H}}
\newcommand{\ci}{{\mathcal I}}

\newcommand{\cm}{{\mathcal M}}

\newcommand{\cs}{{\mathcal S}}

\newcommand{\cz}{{\mathcal Z}}

\newcommand{\fkc}{{\frak c}} \newcommand{\fkC}{{\frak C}}
                             \newcommand{\fkD}{{\frak D}}
 
\newcommand{\fkg}{{\frak g}}
\newcommand{\fkh}{{\frak h}} \newcommand{\fkH}{{\frak H}}

\newcommand{\fkl}{{\frak l}}

\newcommand{\fks}{{\frak s}} \newcommand{\fkS}{{\frak S}}

\newcommand{\scc}{{\mathscr C}}
\newcommand{\scd}{{\mathscr D}}
\newcommand{\sce}{{\mathscr E}}

\newcommand{\scm}{{\mathscr M}}

\newcommand{\bbd}{{\mathbb D}}

\newcommand{\bbn}{{\mathbb N}}
\newcommand{\bbq}{{\mathbb Q}}

\newcommand{\bbz}{{\mathbb Z}}

       \newcommand{\ol}[1]{\overline{#1}}
\newcommand{\ra}{\rightarrow}             
\newcommand{\lan}{{\langle}}              \newcommand{\ran}{{\rangle}}
\newcommand{\geqs}{{\geqslant}}           \newcommand{\leqs}{\leqslant}
\newcommand{\lra}{{\longrightarrow}}

\newcommand{\iso}{\stackrel{_\sim}{\rightarrow}}
\newcommand{\otm}{\otimes}
\newcommand{\bps}{\bigoplus}
\newcommand{\wit}{\widetilde}  \newcommand{\wih}{\widehat}

\newcommand{\llra}{~{\Longleftrightarrow}~}

\newcommand{\ie}{{\em i.e.}~}

\newcommand{\vartri}{\vartriangle}
\newcommand{\U}{{\bf U}}
\newcommand{\bdim}{{\bf dim}}
\newenvironment{psmallmatrix}{\left(\begin{smallmatrix}}{\end{smallmatrix}\right)}

\begin{document}


\title[Ringel-Hall algebras of cyclic quivers]%
{Remarks on PBW bases of Ringel-Hall algebras of cyclic quivers}

\author{Zhonghua Zhao}

\address{Department of Mathematics, Beijing University of Chemical Technology, Beijing 100029, China.}
\email{zhaozh@mail.buct.edu.cn}


\date{\today}


\thanks{The paper was written while the author was visiting the University of New South, the hospitality and support of
UNSW are gratefully acknowledged. He also would like to thank Prof. J. Du for posing the questions and many helpful discussion,
and the China Scholarship Council for the financial support.}

\maketitle

\begin{abstract}
In this paper, we give a recursive formula for the interesting PBW basis $E_{A}$ of composition subalgebras
of Ringel-Hall algebras $\fkH_\vartri(n)$ of cyclic quivers after \cite{DengDuXiao2007generic}, and another
construction of canonical bases of $\U_v^+(\wih{\fks\fkl}_n)$ from the monomial bases $m^{(A)}$
follow \cite{DuZhaomultiplication}. As an application, we will determined all the canonical basis of
$\U_v^+(\wih{\fks\fkl}_2)$ associated with modules of Lowery length $\leqs3$.
Finally, we will discuss the relation of canonical bases of Ringel-Hall algebras and those of affine quantum Schur algebras.
\end{abstract}

\setcounter{tocdepth}{1}
\tableofcontents

\section{Introduction}
Following the remarkable realization \cite{Ringel1990quantum,Ringel1993revisited} of the $\pm$-part of quantum enveloping algebras of finite
type in terms of Hall algebras, Ringel \cite{Ringel1993composition} introduced the generic Hall algebra $\fkH_\vartri(n)$ associated with a
cyclic quiver $\ddz(n)(n\geqs 2)$ and show that its composition subalgebra $\fkC_\vartri(n)$ is isomorphic to the $\pm$-part of the
quantum enveloping algebra $\U_v(\wih{\fks\fkl}_n)$. By Drinfeld's double, the whole quantum enveloping algebra was realized by double Ringel-Hall algebras $\fkD_\vartri(n)$.

Based on the work of Ringel-Hall algebras, another landmark theory was Lusztig's introduction \cite{Lusztig1990canonical} of the canonical basis of the quantum enveloping algebra of a simple complex Lie algebra, and extended to the general cases \cite{Lusztig1991quivers,Lusztig1992affine}. Meanwhile, Kashiwara \cite{Kashiwara1991crystal} introduced the crystal basis theory for arbitrary Kac-Moody algebras and obtained the global crystal basis, which was shown those two bases are coincident \cite{GrojnowskiLusztig1993comparison}. For affine type A, Deng, Du and Xiao \cite{DengDuXiao2007generic} constructed monomial bases and canonical bases of Ringel-Hall algebras and its composition subalgebras algebraically. Moreover, Beck etc. \cite{BeckChariPressley1999algebraic,BeckNakajima2004crystal} gave another algebraic construction of PBW bases and canonical bases in the affine case, but when restricted to affine type A, the relation to the PBW bases in \cite{DengDuXiao2007generic} via quiver representations remains unclear.

Even in the finite case of lower rank, computing canonical bases is in general very difficult, there are only some results of finite type (see, e.g., \cite[\S3]{Lusztig1990canonical} for types $A_1$ and $A_2$ and \cite{XicanonicalA31999, XicanonicalB21999} for type $A_3,B_2$). Partial results about canonical bases of affine $A_1$ were showed by Lusztig in \cite[Sec. 12]{Lusztig1993tight} or \cite[14.5.5]{Lusztig1993introduction}. Recently, Du and the author \cite{DuZhaomultiplication} used the multiplication formulas in \cite{DuFu2015quantum} to compute all the canonical bases associated with modules of Lowey length at most 2 for quantum affine $\mathfrak{gl}_2$.

Schur algebras or their quantum analogue, $q$-Schur algebras, are a class of finite dimensional algebras which play an important role in the theory of Schur-Weyl
duality. The affine analogy of $q$-Schur algebras was given by Ginzberg-Vasserot \cite{GinzburgVasserot1993langlands}. Later, several other version have appeared
in \cite{Lusztig1999aperiodicity,Green1999affine}. In the affine $A$ case, the surjective map $\zeta_r$ from double Ringel-Hall algebras $\fkD_\vartri(n)$ to affine $q$-Schur algebras $\cs_\vartri(n,r)$ was established in \cite{VaragnoloVasserot1999decomposition}, for details, see also \cite{DengDuFu2012double}. The canonical bases of affine $q$-Schur algebras were studied by Lusztig \cite{Lusztig1999aperiodicity} via geometric method. Recently, Du and Fu \cite{DuFu2014integral} give an algebraic construction of those canonical bases.

The main motivation for this paper results from analysing Deng, Du and Xiao's realization of the PBW bases
and canonical bases for the positive part $\U^+$ of quantum affine $\fks\fkl_n$ via generic extension of
Ringel-Hall algebras for cyclic quivers. In \cite{DengDuXiao2007generic}, the authors constructed the PBW basis $E_A$ for the
Lusztig $\cz(=\bbz[v,v^{-1}])$-form $U_\cz^+$ from the strong monomial basis property established in \cite{DengDu2005monomial}, and investigated the
triangular relations of the bar involution on these basis elements. Through a standard linear algebra method, they obtained the canonical basis of quantum affine $\fks\fkl_n$, which agreed with Lusztig's geometric construction of canonical basis in \cite{Lusztig1992affine}. The remarkable property of those PBW basis is following, for $A$ aperiodic, $E_A$ was combined by $\wit{u}_A$ and some $\eta^{C}_A$-linear combination of $\wit{u}_C$ with $C$ periodic and $\eta^{C}_A\in v^{-1}\bbz[v^{-1}]$. In this paper, we first give a
recursive formula of those $\eta^C_A$. Then, we review the construction of the canonical basis of $\fkH_\vartri(n)$ following \cite{DuZhaomultiplication} and give another construction of canonical basis of $\U_v^+(\wih{\fks\fkl}_n)$ from the monomial basis $m^{(A)}$. As an application, we will determine all the canonical basis of $\U_v^+(\wih{\fks\fkl}_2)$ associated with modules of Lowery length $\leqs3$. Finally, we discuss the relation between the canonical bases of Ringel-Hall algebras $\fkH_\vartri(n)$ to those of affine $q$-Schur algebras $\cs_\vartri(n,r)$ through the epimorphism $\zeta_r$. In the appendix, we talk about the tightness of monomials with 3 terms and 4 terms in $\U^+_v(\wih{\fks\fkl}_2)$.

\subsection{Notation}
For a positive integer $n$, let $M_{\vartri,n}(\bbz)$ be the set of all $\bbz\times\bbz$ matrices $A=(a_{i,j})_{i,j\in\bbz}$ with $a_{i,j}\in\bbz$
such that
\begin{itemize}
  \item [(1)] $a_{i,j}=a_{i+n,j+n}$ for $i,j\in\bbz$, and
  \item [(2)] for every $i\in\bbz$, both the set $\{j\in\bbz\mid a_{i,j}\neq0\}$ and $\{j\in\bbz\mid a_{j,i}\neq 0\}$ are finite.
\end{itemize}

Let $\ttz_\vartri(n)=M_{\vartri,n}(\bbn)$ be the subset of $M_{\vartri,n}(\bbz)$ consisting of matrices with entries from $\bbn$,
\begin{equation*}
  \ttz^+_\vartriangle(n)=\{A\in \ttz_\vartriangle(n)\mid a_{ij}=0~\text{for}~i\geqs j\}~\text{and}~\ttz^-_\vartriangle(n)=\{A\in \ttz_\vartriangle(n)\mid a_{ij}=0~\text{for}~i\leqs j\}.
\end{equation*}

For $A\in\ttz_\vartri(n)$, write
\begin{equation*}
  A=A^++A^0+A^-,
\end{equation*}
where $A^0$ is the diagonal submatrix of $A$, $A^+\in\ttz_\vartri^+(n),$ and $A^-\in\ttz_\vartri^-(n)$.

The {\it core} of  a matrix $A$ in $\Theta_\vartri^{+}(n)$ is the $n\times l$ submatrix of $A$ consisting of rows from 1 to $n$ and columns from 1 to $l$, where $l$ is the column index of the right most non-zero entry in the given $n$ rows.

Set $\bbz_{\vartri}^n=\{(\lz_i)_{i\in\bbz}\mid \lz_i\in\bbz,\lz_i=\lz_{i-n}~\text{for}~i\in\bbz\}$ and
$\bbn_{\vartri}^n=\{(\lz_i)_{i\in\bbz}\in\bbz_{\vartri}^n\mid \lz_i\geqs 0~\text{for}~i\in\bbz\}$.
For each $A\in M_{\vartri,n}(\bbz)$,
let
\begin{equation*}
  \row(A)=(\Sigma_{j\in\bbz}a_{i,j})_{i\in\bbz}\in \bbz_{\vartri}^n,\quad \col(A)=(\Sigma_{i\in\bbz}a_{i,j})_{j\in\bbz}\in \bbz_{\vartri}^n.
\end{equation*}

Define an order relation $\leqs$ on $\bbn_\vartri^n$ by $\lz\leqs \mu\llra \lz_i\leqs \mu_i(1\leqs i\leqs n).$
We say $\lz<\mu$ if $\lz\leqs\mu$ and $\lz\neq\mu$.

Let $\bbq(v)$ be the fraction field of $\cz=\bbz[v,v^{-1}]$. For integers $N,t$ with $t\geqs 0$
and $\mu\in\bbz_{\vartri}^n$ and $\lz\in\bbn_{\vartri}^n$,
define Gaussian polynomial and their symmetric version in $\cz$:

\begin{equation*}
  [\![\begin{matrix} t \end{matrix}]\!]!=[\![\begin{matrix} 1 \end{matrix}]\!][\![\begin{matrix} 2 \end{matrix}]\!]\cdots
  [\![\begin{matrix} t \end{matrix}]\!]\quad\text{with}\quad [\![\begin{matrix} m \end{matrix}]\!]=\dfrac{v^{2m}-1}{v^2-1}.
\end{equation*}

\begin{equation*}
  \left[\!\!\left[\begin{matrix} N\\ t \end{matrix}\right]\!\!\right]=\dfrac{[\![\begin{matrix} N \end{matrix}]\!]!}
  {[\![\begin{matrix} t \end{matrix}]\!]![\![\begin{matrix} N-t \end{matrix}]\!]!}=\prod_{1\leqs i\leqs t}\dfrac{v^{2(N-i+1)-1}}{v^{2i}-1},\quad
   \left[\!\!\left[\begin{matrix} \mu\\ \lz \end{matrix}\right]\!\!\right]=\prod_{1\leqs i\leqs n}\left[\!\!\left[\begin{matrix} \mu_i\\ \lz_i \end{matrix}\right]\!\!\right]
   \quad\text{and}\quad
   \begin{bmatrix}
     N\\t
   \end{bmatrix}=v^{-t(N-t)}\left[\!\!\left[\begin{matrix} N\\ t \end{matrix}\right]\!\!\right].
\end{equation*}

\section{The Ringel-Hall algebras of cyclic quivers}
Let $\Delta=\Delta(n)(n\geqs 2)$ be the cyclic quiver
$$
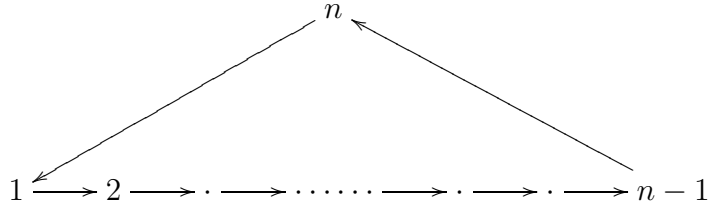
\xymatrix{
&&&n\ar[ddlll]&&&\\
&&&&&&\\
1\ar[r]& 2\ar[r]& \cdot\ar[r]& \cdots\cdots\ar[r] &\cdot\ar[r]& \cdot\ar[r]& n-1\ar[uulll]
}$$
with vertex set $I=\bbz/n\bbz=\{1,2,\cdots,n\}$ and arrow set $\{i\ra i+1\mid i\in I\}$, and $k\Delta$ be the path
algebra of $\Delta$ over a field $k$. For a representation $M=(V_i,f_i)_i$ of $\Delta$,
let ${\bf dim}M=(\dim V_1,\dim V_2,\cdots,\dim V_n)\in\bbn I=\bbn^n$ and $\dim M=\sum\limits_{i=1}^n\dim V_i$
denote the dimension vector and the dimension of $M$, respectively, and let $[M]$ denote the isoclass(isomorphism class) of $M$.

A representation $M=(V_i,f_i)_i$ of $\Delta$ over $k$(or a $k\Delta$-module) is called {\em nilpotent}
if the composition $f_n\cdots f_2f_1:V_1\ra V_1$ is nilpotent, or equivalently,
one of the $f_{i-1}\cdots f_nf_1\cdots f_i:V_i\ra V_i(2\leqs i\leqs n)$ is nilpotent. By $\Rep^0\Delta=\Rep^0_k\Delta(n)$
we denote the category of finite dimensional nilpotent representations of $\Delta(n)$ over $k$. For each vertex $i\in I$, there is a
one-dimensional representation $S_i$ in $\Rep^0\Delta$ satisfying $(S_i)_i=k$ and $(S_i)_j=0$ for $j\neq i$.
It is known that $\{S_i\mid i\in I\}$ form a complete set of simple objects in $\Rep^0\Delta$.

Up to isomorphism, all indecomposable representations in $\Rep^0\Delta$ are given by
$S_i[l](i\in I~\text{and}~l\geqs 1)$ of length $l$ with top $S_i$.

Thus, for any $A=(a_{i,j})\in\ttz_\vartri^+(n)$,
\begin{equation*}
  M(A)=M_k(A)=\bps_{1\leqs i\leqs n,i<j}a_{i,j}S_i[j-i],
\end{equation*}
which means all finite dimensional nilpotent representations of $\ddz(n)$ are indexed by $\ttz_{\vartri}^+(n)$.

A matrix $A=(a_{i,j})\in\ttz_{\vartri}^+(n)$ is called {\em aperiodic} if, for each $l\geqs 1$, there exists $i\in\bbz$ such that $a_{i,i+l}=0$.
Otherwise, $A$ is called {\em periodic}. Denote by $\ttz_\vartri^{ap}(n)(\text{resp.}~\ttz_\vartri^{p}(n))$
the set of all aperiodic(resp. periodic) matrix in $\ttz_\vartri^+(n)$.  A nilpotent representation $M(A)$ is
called {\it aperiodic(resp. periodic)} if $A$ is aperiodic(resp. periodic).

For ${\bf a}=(a_i)\in\bbz_\vartri^n$ and ${\bf b}=(b_i)\in\bbz_\vartri^n$, the Euler form associated
with the cyclic quiver $\ddz(n)$ is the bilinear form $\lan-,-\ran:\bbz_\vartri^n\times\bbz_\vartri^n\lra\bbz$ defined by
\begin{equation*}
  \lan {\bf a,b}\ran=\sum_{i\in I}a_ib_i-\sum_{i\in I}a_ib_{i+1}.
\end{equation*}

Let $k$ be a {\em finite} field of $q_k$ elements and, for $A,B,C\in \ttz_{\vartri}^+(n)$, let $\fkH_{M_k(B),M_k(C)}^{M_k(A)}$
be the number of submodules $N$ of $M_k(A)$ such that
$N\cong M_k(C)$ and $M_k(A)/N\cong M_k(B)$. More generally, given $A,B_1,B_2,\cdots,B_m\in\ttz_{\vartri}^+(n)$, denote by $\fkH_{M_k(B_1),M_k(B_2),\cdots,M_k(B_m)}^{M_k(A)}$
the number of filtrations
\begin{equation*}
  M_0=M_k(A)\supseteq M_1\supseteq M_2\supseteq \cdots\supseteq M_{m-1}\supseteq M_m=0,
\end{equation*}
such that $M_{t-1}/M_t\cong M_k(B_t)$ for $1\leqs t\leqs m$. By \cite{Guo1995hallpoly,Ringel1993composition},
there is a polynomial $\vz_{B_1,B_2,\cdots,B_m}^A\in \bbz[v^2]$ in $v^2$, called the {\em Hall polynomials}\footnote{Using multiplications formulas, Du and the author have reproved the existence of Hall polynomials and a recursive formula was given via the degenerated order, for details, see \cite{DuZhaomultiplication}.}, such that for the finite field $k$, $$\vz_{B_1,B_2,\cdots,B_m}^A|_{v^2=q_k}=\fkH_{M_k(B_1),M_k(B_2),\cdots,M_k(B_m)}^{M_k(A)}.$$

The generic(twisted) Ringel-Hall algebra $\fkH_{\vartri}(n)$ of $\ddz(n)$ is by definition the $\cz$-algebra with basis $\{u_A=u_{[M(A)]}\mid A\in \ttz_{\vartri}^+(n)\}$ and multiplication given by
\begin{equation*}
  u_{B}u_{C}=v^{\lan{\bf dim} M(B),{\bf dim} M(C)\ran}\sum_{A\in\ttz_{\vartri}^+(n)}\vz^A_{B,C}({v^2})u_A
\end{equation*}

It is well known that for $A,B\in\ttz_\vartri^+(n)$, there holds
$$\lan{\bf dim} M(A),{\bf dim} M(B)\ran=\dim_k\Hom(M(A),M(B))-\dim_k\Ext^1(M(A),M(B)).$$

The $\cz$-subalgebra $\mathfrak{C}_\vartri(n)$ of $\fkH_\vartri(n)$ generated by $u_i^{(m)}=\dfrac{u_i^m}{[m]!},i\in I$ and $m\geqs 1$ called the (twisted){\em composition subalgebra}. Then $\fkC_\vartri(n)$ is also generated by $u_{[mS_i]},i\in I,m\geqs 1$ since $u_i^{(m)}=v^{m(m-1)}u_{[mS_i]}$. Clearly, $\fkH_\vartri(n)\text{and}~\fkC_\vartri(n)$ admit natural $\bbn^n$-grading by dimension vectors:
$$\fkH_\vartri(n)=\bps_{{\bf d}\in\bbn^n}\fkH_\vartri(n)_{\bf d}\quad\text{and}\quad\fkC_\vartri(n)=\bps_{{\bf d}\in\bbn^n}\fkC_\vartri(n)_{\bf d},$$
where $\fkH_\vartri(n)_{\bf d}$ is spanned by all $u_A$ with ${\bf dim} M(A)={\bf d}$ and $\fkC_\vartri(n)_{\bf d}=\fkC_\vartri(n)\cap \fkH_\vartri(n)_{\bf d}$.

Base change gives the $\bbq(v)$-algebra ${\bm\fkH}_\vartri(n)=\fkH_\vartri(n)\otm_\cz\bbq(v)$ and ${\bm\fkC}_\vartri(n)=\fkC_\vartri(n)\otm_{\cz}\bbq(v)$.

Denote by ${\bm\fkH}^-_\vartri(n)$ the opposite algebra of ${\bm\fkH}^+_\vartri(n)(={\bm\fkH}_\vartri(n))$. By extending $\bm\fkH_\vartri(n)$ to Hopf algebras
\begin{equation*}
  \bm\fkH_\vartri(n)^{\geqs0}={\bm\fkH}^+_\vartri(n)\otm\bbq(v)[K_1^{\pm1},\cdots,K_n^{\pm1}]~
  \text{and}~\bm\fkH_\vartri(n)^{\leqs0}=\bbq(v)[K_1^{\pm1},\cdots,K_n^{\pm1}]\otm {\bm\fkH}^-_\vartri(n),
\end{equation*}
we define the double Ringel-Hall algebra $\fkD_\vartri(n)$(cf. \cite{Xiao1997drinfeld} \& \cite{DengDuFu2012double})
to be a quotient algebra of the free product $\bm\fkH_\vartri(n)^{\geqs0}*\bm\fkH_\vartri(n)^{\leqs0}$ via a certain skew Hopf paring $\psi:\bm\fkH_\vartri(n)^{\geqs0}\times\bm\fkH_\vartri(n)^{\leqs0}\ra\bbq(v)$. In particular, there is a triangular decomposition
$$\fkD_\vartri(n)=\fkD_{\vartri}^+(n)\otm \fkD_\vartri^0(n)\otm \fkD_\vartri^-(n),$$
where $\fkD_\vartri^+(n)={\bm\fkH}^+_\vartri(n),\fkD_\vartri^0(n)=\bbq[K_1^{\pm1},\cdots,K_n^{\pm1}]$ and $\fkD_\vartri^-(n)={\bm\fkH}^-_\vartri(n)$.

\begin{thm}[{\cite[Thm 2.5.3]{DengDuFu2012double}}]\label{thm 1}
  Let $\U_v(\wih{\fkg\fkl}_n)$ be the quantum enveloping algebra of the loop algebra of $\fkg\fkl_n$ defined in \cite{Drinfeld1988new} or \cite[\S2.5]{DengDuFu2012double}. Then there is a Hopf algebra isomorphism $\fkD_\vartri(n)\cong \U_v(\wih{\fkg\fkl}_n)$.
\end{thm}

\section{Generic extension and distinguished words}

Let $\cm$ be the set of all isoclasses of representation in $\Rep^0\ddz$.
Given two objects $M,N\in\Rep^0\ddz$, there exists a unique(up to isomorphism) extension $G$ of $M$ by $N$ with minimal $\dim\End(G)$\cite{Bongartz1996degenerations,Reineke2001generic,DengDu2005monomial,DengDuParashallWang2008finite}.
The extension $G$ is called the {\em generic extension}\footnote{When the field $k$ is algebraically closed,
see \cite{Reineke2001generic} for a geometrical description.} of $M$ by $N$ and is
denoted by $G=M*N$. If we define $[M]*[N]=[M*N]$, then it is known form \cite{Reineke2001generic} that $*$ is
associative and $(\cm,*)$ is a monoid with identity $[0]$.

Besides the monoid structure, $\cm$ has also a poset structure. For two nilpotent representations $M,N\in\Rep^0\ddz$ with
${\bf dim} M={\bf dim} N$, define
$$N\leqs_{\dg}M\llra \dim\Hom(X,N)\geqs\dim\Hom(X,M),~\text{for all}~X\in\Rep^0\ddz.$$
see \cite{Zwara1997degenerations}. This gives rise to a partial order on the set of isoclasses of representations in $\Rep^0\ddz$,
called the {\em degeneration order}. Thus, it also induces a partial order on $\ttz_\vartri^+(n)$ by setting
$$A\leqs_{\dg} B\llra M(A)\leqs_{\dg} M(B).$$

Following \cite{BeilinsonLusztigMacPherson1990geometric} we define the order relation
$\preccurlyeq$ on $M_{\vartri,n}(\bbz)$ as follows. For $A\in M_{\vartri,n}(\bbz)$ and $i\neq j\in\bbz$, let
\begin{equation*}\sz_{i,j}(A)=
  \begin{cases}
    \sum\limits_{s\leqs i,t\geqs j}a_{s,t},&\text{if}~i<j,\\
    \sum\limits_{s\geqs i,t\leqs j}a_{s,t},&\text{if}~i>j.
  \end{cases}
\end{equation*}

For $A,B\in M_{\vartri,n(\bbz)}$, define
$$B\preccurlyeq A~\text{if and only if~}\sz_{i,j}(B)\leqs\sz_{i,j}(A)~\text{for all}~i\neq j.$$

Set $B\prec A$ if $B\preccurlyeq A$, and for some $(i,j)$ with $i\neq j,\sz_{i,j}(B)<\sz_{i,j}(A)$.

It is shown in \cite[Thm 6.2]{DuFu2010modified} that, if $A,B\in\ttz_{\vartri}^+(n)$, then

\begin{equation}\label{equivalent condition for degenerate order}
  B\leqs_{\dg}A\llra B\preccurlyeq A,~\bdim M(A)=\bdim M(B).
\end{equation}

An element $\lz\in\bbn_\vartri^n$ is called {\em sincere} if $\lz_i>0$ for all $i\in I$. Let $\wit{I}=I\cup\{\text{all sincere vectors in } \bbn_\vartri^n\}$, and $\wit{\ssz}(\text{resp.}~\ssz)$ be the set of words on the alphabet $\wit{I}(\text{resp.}~I)$.
For each $w={\bm a}_1{\bm a}_2\cdots{\bm a}_m\in\wit{\ssz}$, we set $M(w)=S_{{\bm a}_1}*S_{{\bm a}_2}*\cdots*S_{{\bm a}_m}$.
Then there is a unique $A\in\ttz_\vartri^+(n)$ such that $M(w)\cong M(A)$, and we set $\wp(w)=A$, which induces a surjective map $\wp:\wit{\ssz}\lra\ttz_\vartri^+(n),w\mapsto \wp(w)$. Note that $\wp$ induces a surjective map $\wp:\ssz\lra\ttz_\vartri^{ap}(n)$.

For $\bm a\in\wit{I}$, set $u_{\bm a}=u_{[S_{\bm a}]}$. Let $w={\bm a}_1{\bm a}_2\cdots{\bm a}_m\in\wit{\ssz}$ and $\vz^A_w(v^2)$ be the Hall polynomial $\vz^A_{B_1,B_2,\cdots,B_m}(v^2)$ with $M(B_i)\cong S_{\bm a_i}$. Any word $w={\bm a}_1{\bm a}_2\cdots{\bm a}_m\in\wit{\ssz}$ can be uniquely expressed in the {\em tight form} $w={\bm b}_1^{e_1}{\bm b}_2^{e_2}\cdots{\bm b}_t^{e_t}$ where $e_i=1$ if ${\bm b_i}$ is sincere, and $e_i$ is the number of consecutive occurrence of $\bm b_i$ if $\bm b_i\in I$. A filtration
$$M=M_0\supseteq M_1\supseteq \cdots\supseteq M_{t-1}\supseteq M_t=0$$ of nilpotent representation $M$ is called a {\em reduced} filtration of type $w$ if
$M_{r-1}/M_r\cong e_rS_{\bm b_r}$ for all $1\leqs r\leqs t$.
Denote by $\gz^A_{w}(v^2)$ the Hall polynomial $\vz^A_{B_1,B_2,\cdots,B_t}(v^2)$, where
$M(B_r)=e_rS_{\bm b_r}$. Thus, for any finite field $k$ with $q_k$ elements, $\gz^A_{w}(q_k)$ is the number of the reduced filtrations of $M(A)$ of type $w$.
A word $w$ is called {\em distinguished} if the Hall polynomial $\gz^{\wp(w)}_w(v^2)=1$.

{\bf Sometimes, writing $\gz^{A}_B=\gz^A_{w_B}$ means we have fixed a distinguished word $w_B\in\wp^{-1}(B)$ ahead of time.}

For $A\in\ttz_\vartri^+(n)$, denote by $\ell(A)=\ell(M(A))$ the {\em Loewy length} of $M(A)$ and
$$p(A)=\max\{l\mid a_{i,i+l}\neq 0~\text{for all}~1\leqs i\leqs n,l\in\bbn\}.$$

If no such $p(A)$ exists, set $p(A)=0$, in this case, $A$ is aperiodic. $A$ is called {\em strongly periodic} if $p(A)=\ell(A)$.

\begin{thm}[\cite{DengDuXiao2007generic}]
\begin{enumerate}[\rm(1)]
\item For any $A\in\ttz_\vartri^+(n)$, there exists uniquely a pair $(A_1,A_2)$ associated with $A$ such that $A_1$ is aperiodic, $A_2$ is strongly periodic and $M(A)\cong M(A_1)* M(A_2)$.

\item For aperiodic part $A_1$, there exists a distinguished word $w_{A_1}=j_1^{e_1}j_2^{e_2}\cdots j_t^{e_t}\in \ssz\cap \wp^{-1}(A_1)$.

\item For strongly periodic part $A_2$, there exists a distinguished word $w_{A_2}=\bm a_1\bm a_2\cdots \bm a_p\in \wit{\ssz}\cap \wp^{-1}(A_2)$,
  moreover, $S_{\bm a_s}\cong \soc^{p-s+1}M(A_2)/\soc^{p-s}M(A_2),1\leqs s\leqs p=p(A)$.

\item $w_{A_1}w_{A_2}=j_1^{e_1}j_2^{e_2}\cdots j_t^{e_t}\bm a_1\bm a_2\cdots \bm a_p$ is a distinguished word of $A$.
\end{enumerate}
\end{thm}
\begin{rmk}
  A matrix algorithm was showed in \cite{DuZhaomultiplication} for taking distinguished words.
\end{rmk}

\section{A recursive formula for PBW basis $E_A$}\label{1}
Let $\U=\U_v(\wih{\fks\fkl}_n)$ be the quantum affine $\fks\fkl_n(n\geqs 2)$ over $\bbq(v)$, and let $E_i,F_i,K_i^\pm(i\in I)$ be the generators, for
details see \cite{Lusztig1993introduction,Jantzen1995lectures}. Then $\U$ admits a triangular decomposition $\U=\U^-\U^0\U^+$, where $\U^+$(resp. $\U^-,\U^0$) is the subalgebra generated by the $E_i$(resp. $F_i$, $K_i^\pm~(i\in I)$). Denote by $U_\cz^+$ the Lusztig integral form of $\U^+$, which is
generated by all the divided powers $E_i^{(m)}=\tfrac{E_i^m}{[m]!}$. The relation of Ringel-Hall algebras and quantum affine $\fks\fkl_n$ is described in the following.

\begin{thm}[\cite{Ringel1993composition}]\label{isomorphism theorem for composition algebra}
There is a $\cz$-algebra isomorphism $$\fkC_\vartri(n)\iso U_\cz^+,~u_i^{(m)}\mapsto E_i^{(m)},~i\in I,~m\geqs 1,$$
  and by base change to $\bbq(v)$, there is an algebra isomorphism $\bm\fkC_\vartri(n)\iso \U^+$.
\end{thm}

For $A\in\ttz_\vartri^+(n)$, let $\dz(A)=\dim\End(M(A))-\dim M(A)$ and
$$\wit{u}_A=v^{\dz(A)}u_{A}=v^{\dim\End(M(A))-\dim M(A)}u_{A}.$$

For each $w={\bm a}_1{\bm a}_2\cdots{\bm a}_m\in\wit{\ssz}$ with tight form $w={\bm b}_1^{e_1}{\bm b}_2^{e_2}\cdots{\bm b}_t^{e_t}$,
define the monomial associated with $w$ as
$$m^{(w)}=\wit{u}_{e_1\bm b_1}\cdots \wit{u}_{e_t\bm b_t}.$$

For each $A\in\ttz_\vartri^+(n)$, pick $w_A\in\wp^{-1}(A)\cap\wit{\ssz}$, $\bbd=\{w_A\mid A\in\ttz_\vartri^+(n)\}$ is called a {\em section} of $\wit{\ssz}$ over $\ttz_\vartri^+(n)$. A section is called a {\em distinguished section} of $\wit{\ssz}$ over $\ttz_\vartri^+(n)$ if all $w_A$ are chosen to be distinguished word. When restricted to $\ttz_\vartri^{ap}(n)$, we obtain a distinguished section of $\ssz$ over $\ttz_\vartri^{ap}(n)$.

\begin{thm}[\cite{DengDuXiao2007generic}]
\begin{enumerate}[\rm(1)]
\item For $A\in\ttz_\vartri^+(n)$, take $w_A\in\wp^{-1}(A)$, we have a triangular relation
 \begin{equation}\label{triangluarrelation}
  m^{(A)}=m^{(w_A)}=\wit{u}_A
  +\sum_{\stackrel{T\prec A,T\in\ttz_\vartri^+(n)}{\bdim M(A)=\bdim M(T)}}v^{\dz(A)-\dz(T)}\gz^T_A(v^2)\wit{u}_T,
 \end{equation}
In particular, $\fkH_\vartri(n)$ is generated by $\{u_\lz=u_{[S_\lz]}\mid \lz\in\bbn_\vartri^n\}$,
where $S_\lz=\oplus_{i=1}^n\lz_i S_i$ is the semisimple representation of $\ddz(n)$.

\item For a given distinguished section $\bbd=\{w_A\mid A\in\ttz_\vartri^+(n)\}$ of $\wit{\ssz}$ over $\ttz_\vartri^+(n)$, $\{m^{(A)}\mid A\in\ttz_\vartri^+(n)\}$ is a $\cz$-basis of $\fkH_\vartri(n)$ and $\{m^{(A)}\mid A\in\ttz_\vartri^{ap}(n)\}$ is a $\cz$-basis of $U_\cz^+$.
\end{enumerate}
\end{thm}

In order to obtain the canonical basis of $U_\cz^+$, Deng, Du and Xiao gave a construction of PBW basis $E_A$.

\begin{defn}
  For the given distinguished section $\bbd=\{w_A\mid A\in\ttz_\vartri^+(n)\}$, any ${\bf d}\in\bbn^n_\vartri$
  and $A\in\ttz_\vartri^{ap}(n)_{\bf d}=\{A\in\ttz_\vartri^{ap}(n)\mid\bdim M(A)=\bf d\}$, if $A$ is minimal under $\preccurlyeq$, put $E_A=m^{(A)}$.
  Otherwise, put
  $$E_A=m^{(A)}-\sum_{B\in\ttz_\vartri^{ap}(n)_{\bf d},B\prec A}v^{\dz(A)-\dz(B)}\gz^B_{A}(v^2)E_B.$$
\end{defn}

Or equivalently, we have
$$m^{(A)}=E_A+\sum_{B\in\ttz_\vartri^{ap}(n)_{\bf d},B\prec A}v^{\dz(A)-\dz(B)}\gz^B_{A}(v^2)E_B.$$

The interesting property of $E_A$ is the following
\begin{thm}[\cite{DengDuXiao2007generic}]
\begin{enumerate}[\rm(1)]
\item Let $\bbd=\{w_A\mid A\in\ttz_\vartri^+(n)\}$ be a given distinguished section. For each ${\bf d}\in\bbn^n_\vartri$ and $A\in\ttz_\vartri^{ap}(n)_{\bf d}$, we have
    \begin{equation}\label{formula of E_A}
      E_A=\wit{u}_A+\sum_{C\in\ttz_\vartri^{p}(n)_{\bf d},C\prec A}\eta^{C}_A\wit{u}_C,\quad\text{where}~~\eta^C_A\in v^{-1}\bbz[v^{-1}].
    \end{equation}
\item The PBW basis $\{E_A\mid A\in\ttz_\vartri^{ap}(n)\}$ is independently of the selection of distinguished section.
\end{enumerate}
\end{thm}

First of all, we give another construction of $E_A$ and show the uniqueness. For $A\in\ttz_{\vartri}^{ap}(n)_{\bf d},~{\bf d}\in\bbn^n_\vartri$, define
$$\sce_A=m^{(A)}-\sum_{\stackrel{B\prec A}{B\in\ttz_{\vartri}^{ap}(n)}}f_{B,A}m^{(B)}$$
for suitable $f_{B,A}\in\cz$ such that $\sce_A=\wit{u}_A+\sum_{\stackrel{D\prec A}{D\in\ttz_{\vartri}^{p}(n)}}g_{D,A}\wit{u}_D$ with $g_{D,A}\in\cz$.

\begin{lem}[{\cite[Lem 7.1]{DengDuXiao2007generic}}]\label{decomp hall alg as v.s}
  Let $\bf P$ be the subspace of $\bm\fkH_\vartri(n)$ spanned by all $u_A$ with $A\in\ttz_{\vartri}^{p}(n)$. Then as a vector
  space $\bm\fkH_\vartri(n)={\bf P}\oplus\U^+$.
\end{lem}

Now the uniqueness of PBW basis $E_A$ is given as follows.

\begin{lem}
  $\sce_A=E_A$. Thus, the construction of $E_A$ is unique respect to $m^{(A)}$ and $\preccurlyeq$.
\end{lem}
\begin{proof}
  $\sce_A-E_A=\sum\limits_{\stackrel{B'\prec A}{B'\in\ttz_{\vartri}^{p}(n)}}f_{B',A}\wit{u}_{B'}\in \U^+\cap {\bf P}=0$ by Lemma \ref{decomp hall alg as v.s}.
\end{proof}

\begin{thm}\label{recursive formula of coeff of E_A}For $A\in\ttz_{\vartri}^{ap}(n)_{\bf d},~{\bf d}\in\bbn^n_\vartri$, fix a set $$\Pi_{\bf d}^{\prec A}=\{B\in\ttz_{\vartri}^+(n)_{\bf d}\mid B\prec A\}=\{\underbrace{B_1,\cdots,B_k}_{\in\ttz_\vartri^{ap}(n)},\underbrace{C_1,\cdots,C_t}_{\in\ttz_\vartri^{p}(n)}\},$$ then the coefficient $\eta^{C}_A$ of $E_A$ in \eqref{formula of E_A}
has a recursive formula as
    \begin{equation*}
   \eta^{C_i}_A=v^{\dz(A)-\dz(C_i)}\gz^{C_i}_A(v^2)-\sum_{j=1}^kv^{\dz(A)-\dz(B_j)}\gz^{B_j}_A(v^2)\eta^{C_i}_{B_j}.
    \end{equation*}
    Here we understand $\eta^{C_i}_{B_j}=0$ if $C_i,B_j$ are not comparable. Thus, if $\eta^{C_i}_{B_j}\neq0$, then $C_i\prec B_j$.
\end{thm}
\begin{proof}By the triangular decomposition in \eqref{triangluarrelation}, we have
  \begin{equation*}
    m^{(A)}=\wit{u}_A+\sum_{i=1}^kv^{\dz(A)-\dz(B_i)}\gz^{B_i}_A(v^2)\wit{u}_{B_i}+\sum_{j=1}^tv^{\dz(A)-\dz(C_j)}\gz^{C_j}_A(v^2)\wit{u}_{C_j}.
  \end{equation*}

On the other hand, express $m^{(A)}$ as
\begin{equation*}
  \begin{split}
    m^{(A)}&=E_A+\sum_{B_i\prec A}f_{B_i,A}E_{B_i}\quad (\text{Here}~f_{B_i,A}=v^{\dz(A)-\dz(B_i)}\gz^{B_i}_A(v^2))\\
    &=(\wit{u}_A,\wit{u}_{C_1},\cdots,\wit{u}_{C_t})\begin{pmatrix}
      1\\ \eta^{C_1}_A\\\eta^{C_2}_A\\ \vdots \\\eta^{C_t}_A
    \end{pmatrix}+\sum_{B_i\prec A}f_{B_i,A}(\wit{u}_{B_i},\wit{u}_{C_1},\cdots,\wit{u}_{C_t})\begin{pmatrix}
      1\\ \eta^{C_1}_{B_i}\\\eta^{C_2}_{B_i}\\ \vdots \\\eta^{C_t}_{B_i}\end{pmatrix}\\
      &=\wit{u}_A+f_{B_1,A}\wit{u}_{B_1}+f_{B_2,A}\wit{u}_{B_2}+\cdots+f_{B_k,A}\wit{u}_{B_k}\\
      &\quad+(\eta^{C_1}_A+f_{B_1,A}\eta^{C_1}_{B_1}+\cdots+f_{B_k,A}\eta^{C_1}_{B_k})\wit{u}_{C_1}\\
      &\quad+(\eta^{C_2}_A+f_{B_2,A}\eta^{C_2}_{B_1}+\cdots+f_{B_k,A}\eta^{C_2}_{B_k})\wit{u}_{C_2}\\
      &\quad+\quad\cdots\cdots\cdots\\
      &\quad+(\eta^{C_t}_A+f_{B_1,A}\eta^{C_t}_{B_1}+\cdots+f_{B_k,A}\eta^{C_t}_{B_k})\wit{u}_{C_t},\\
  \end{split}
\end{equation*}

Comparing the coefficients on both expression, the result follows.
\end{proof}

Take the same example as in \cite{DengDuXiao2007generic}.

\begin{exam}\label{example 1}
  Let $n=3$ and ${\bf d}=(1,2,3)$. Then $\ttz_{\vartri}^+(3)_{\bf d}$ consists of 18 elements, \ie there are 18 isoclasses of nilpotent representations of $\ddz(3)$ of dimension vector $\bf d$. The Hasse diagram of $(\ttz_{\vartri}^+(3)_{\bf d},\preccurlyeq)$ has the form
 $$
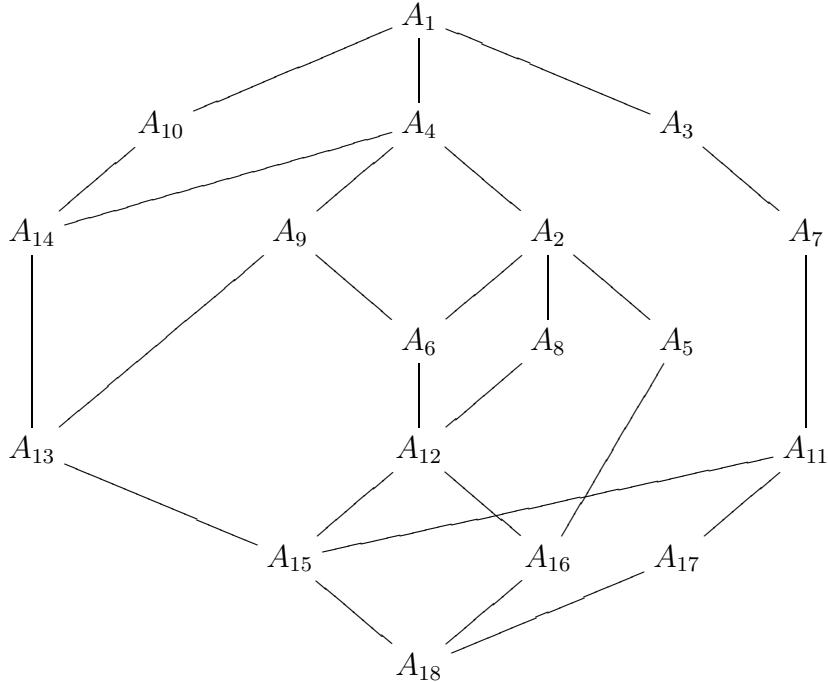
\xymatrix{
 & &&A_1\ar@{-}[dll]\ar@{-}[d]\ar@{-}[drr]& &\\
 &A_{10}\ar@{-}[dl]& & A_4\ar@{-}[dlll]\ar@{-}[dl]\ar@{-}[dr] & & A_3\ar@{-}[dr] \\
 A_{14}\ar@{-}[dd]&& A_9\ar@{-}[ddll]\ar@{-}[dr]& & A_2\ar@{-}[dl]\ar@{-}[d]\ar@{-}[dr] && A_7\ar@{-}[dd]\\
&&& A_6\ar@{-}[d]  & A_8\ar@{-}[dl] &  A_5\ar@{-}[ddl]\\
 A_{13}\ar@{-}[drr] && &A_{12}\ar@{-}[dl]\ar@{-}[dr]&& &A_{11}\ar@{-}[dllll]\ar@{-}[dl]&\\
 & &A_{15}\ar@{-}[dr]&& A_{16}\ar@{-}[dl] & A_{17}\ar@{-}[dll]\\
  &&& A_{18}
 }$$

where the cores of those matrices are given by

\begin{align*}
      &A_1=\begin{psmallmatrix}
      \bf0 & 0 & 0 & 0 & 0 & 0 & 0\\
      0 & \bf0 & 0 & 0 & 0 & 0 & 1\\
      0 & 0 & \bf0 & 1 & 0 & 0 & 0\\
      \end{psmallmatrix},~
      A_2=\begin{psmallmatrix}
      \bf0 & 0 & 0 & 0 & 0 & 0 & 0\\
      0 & \bf0 & 1 & 0 & 0 & 0 & 0\\
      0 & 0 & \bf0 & 1 & 0 & 0 & 1\\
      \end{psmallmatrix},~
      A_3=\begin{psmallmatrix}
      \bf0 & 0 & 0 & 0 & 0 & 0 & 0\\
      0 & \bf0 & 0 & 0 & 0 & 0 & 0\\
      0 & 0 &\bf 0 & 2 & 0 & 0 & 1\\
      \end{psmallmatrix},~
      A_4=\begin{psmallmatrix}
      \bf0 & 0 & 0 & 0 & 0 & 0 & 0\\
      0 & \bf0 & 0 & 1 & 0 & 0 & 0\\
      0 & 0 & \bf0 & 0 & 0 & 0 & 1\\
      \end{psmallmatrix},\\
      &A_5=\begin{psmallmatrix}
      \bf0 & 0 & 0 & 0 & 0 & 0\\
      0 & \bf0 & 1 & 0 & 0 & 0\\
      0 & 0 & \bf0 & 2 & 0 & 1\\
      \end{psmallmatrix},~
      A_6=\begin{psmallmatrix}
      \bf0 & 0 & 0 & 1\\
      0 & \bf0 & 1 & 0\\
      0 & 0 & \bf0 & 2 \\
      \end{psmallmatrix},~
      A_7=\begin{psmallmatrix}
      \bf0 & 0 & 0 & 0 & 0 & 0 \\
      0 & \bf0 & 1 & 0 & 1 & 0 \\
      0 & 0 &\bf 0 & 2 & 0 & 1\\
      \end{psmallmatrix},~
      A_8=\begin{psmallmatrix}
      \bf0 & 0 & 0 & 0 & 0 & 0 \\
      0 & \bf0 & 0 & 1 & 0 & 0 \\
      0 & 0 &\bf 0 & 1 & 0 & 1\\
      \end{psmallmatrix},\\
      &A_9=\begin{psmallmatrix}
      \bf0 & 0 & 0 & 1 \\
      0 &\bf 0 & 0 & 1\\
      0 & 0 & \bf0 & 1\\
      \end{psmallmatrix},~
      A_{10}=\begin{psmallmatrix}
      \bf0 & 0 & 0 & 0 & 0 \\
      0 & \bf0 & 0 & 1 & 1 \\
      0 & 0 &\bf 0 & 1 & 0 \\
      \end{psmallmatrix},~
      A_{11}=\begin{psmallmatrix}
      \bf0 & 0 & 0 & 0 & 0 \\
      0 & \bf0 & 1 & 1 & 0 \\
      0 & 0 & \bf0 & 1 & 1\\
      \end{psmallmatrix},~
      A_{12}=\begin{psmallmatrix}
      \bf0 & 0 & 1 & 0 \\
      0 & \bf0 & 0 & 1 \\
      0 & 0 &\bf 0 & 2 \\
      \end{psmallmatrix},\\
      &A_{13}=\begin{psmallmatrix}
      \bf0 & 1 & 0 & 0 \\
      0 & \bf0 & 0 & 2 \\
      0 & 0 &\bf 0 & 1 \\
      \end{psmallmatrix},~
      A_{14}=\begin{psmallmatrix}
      \bf0 & 0 & 0 & 0 & 0 \\
      0 & \bf0 & 0 & 2 & 0 \\
      0 & 0 &\bf 0 & 0 & 1 \\
      \end{psmallmatrix},~
      A_{15}=\begin{psmallmatrix}
      \bf0 & 1 & 0 & 0 \\
      0 & \bf0 & 1 & 1 \\
      0 & 0 & \bf0 & 2 \\
      \end{psmallmatrix},~
      A_{16}=\begin{psmallmatrix}
      \bf0 & 0 & 1 & 0\\
      0 & \bf0 & 1 & 0\\
      0 & 0 & \bf0 & 3\\
      \end{psmallmatrix},\\
      &A_{17}=\begin{psmallmatrix}
      \bf0 & 0 & 0 & 0 & 0 \\
      0 & \bf0 & 2 & 0 & 0 \\
      0 & 0 & \bf0 & 2 & 1 \\
      \end{psmallmatrix},~
      A_{18}=\begin{psmallmatrix}
      \bf0 & 1 & 0 & 0 \\
      0 & \bf0 & 2 & 0 \\
      0 & 0 & \bf0 & 3 \\
      \end{psmallmatrix}.
\end{align*}

Easy to see, all are aperiodic except $A_{15},A_{18}$.

For $\Pi_{\bf d}^{\preccurlyeq A_6}=\{A_6,A_{12},A_{16},A_{15},A_{18}\}$, using the matrix algorithm in \cite{DuZhaomultiplication},
we take the following distinguished words $123^32\in\wp^{-1}(A_6),213^32\in\wp^{-1}(A_{12})$, $13^32^2\in\wp^{-1}(A_{16})$.

By easy calculation, $\dz(A_6)=2,\dz(A_{12})=3,\dz(A_{16})=6,\dz(A_{15})=4,\dz(A_{18})=8$ and
 $\gz^{A_{18}}_{A_{16}}(v^2)=\gz^{A_{16}}_{A_{12}}(v^2)=\gz^{A_{15}}_{A_{12}}(v^2)=\gz^{A_{15}}_{A_{6}}(v^2)=\gz^{A_{12}}_{A_{6}}(v^2)=1$, $\gz^{A_{18}}_{A_{12}}(v^2)=\gz^{A_{18}}_{A_6}(v^2)=\gz^{A_{16}}_{A_6}(v^2)=1+v^2$.

 Since $A_{16}$ is minimal, $\eta^{A_{18}}_{A_{16}}=v^{\dz(A_{16})-\dz(A_{18})}\gz^{A_{18}}_{A_{16}}(v^2)=v^{-2}$.

 $A_{12}$ is minimal in $\Pi_{\bf d}^{\preccurlyeq A_6}\setminus\{A_{16}\}$, by Theorem \ref{recursive formula of coeff of E_A}, we have
 \begin{equation*}
   \begin{split}
     \eta^{A_{18}}_{A_{12}}&=v^{\dz(A_{12})-\dz(A_{18})}\gz^{A_{18}}_{A_{12}}(v^2)
     -v^{\dz(A_{12})-\dz(A_{16})}\gz^{A_{16}}_{A_{12}}(v^2)\eta^{A_{18}}_{A_{16}}\\
     &=v^{3-8}(1+v^2)-v^{3-6}\cdot 1\cdot v^{-2}=v^{-3},\\
     \eta^{A_{15}}_{A_{12}}&=v^{\dz(A_{12})-\dz(A_{15})}\gz^{A_{15}}_{A_{12}}(v^2)
     -v^{\dz(A_{16})-\dz(A_{15})}\gz^{A_{15}}_{A_{16}}(v^2)\eta^{A_{15}}_{A_{16}}\qquad~(A_{15},A_{16}~\text{are not comparable})\\
     &=v^{-1}-0=v^{-1},\\
     \eta^{A_{18}}_{A_{6}}&=v^{\dz(A_{6})-\dz(A_{18})}\gz^{A_{18}}_{A_{6}}(v^2)
     -v^{\dz(A_{6})-\dz(A_{12})}\gz^{A_{12}}_{A_{6}}(v^2)\eta^{A_{18}}_{A_{12}}-v^{\dz(A_{6})
     -\dz(A_{16})}\gz^{A_{16}}_{A_{6}}(v^2)\eta^{A_{18}}_{A_{16}}\\
     &=v^{2-8}(1+v^2)-v^{2-3}\cdot1\cdot v^{-3}-v^{2-6}(1+v^2)v^{-2}=-v^{-4},\\
     \eta^{A_{15}}_{A_{6}}&=v^{\dz(A_{6})-\dz(A_{15})}\gz^{A_{15}}_{A_{6}}(v^2)
     -v^{\dz(A_{6})-\dz(A_{12})}\gz^{A_{12}}_{A_{6}}(v^2)\eta^{A_{15}}_{A_{12}}-v^{\dz(A_{6})
     -\dz(A_{16})}\gz^{A_{16}}_{A_{6}}(v^2)\eta^{A_{15}}_{A_{16}}\\
     &=v^{2-4}\cdot 1-v^{2-3}\cdot1\cdot v^{-1}-0=0.
   \end{split}
 \end{equation*}

 Finally, we obtain
 \begin{equation*}
   \begin{split}
     E_{A_6}&=\wit{u}_{A_6}-v^{-4}\wit{u}_{A_{18}},\\
     E_{A_{12}}&=\wit{u}_{A_{12}}+v^{-1}\wit{u}_{A_{15}}+v^{-3}\wit{u}_{A_{18}},\\
     E_{A_{16}}&=\wit{u}_{A_{16}}+v^{-2}\wit{u}_{A_{18}}.
   \end{split}
 \end{equation*}
\end{exam}

\section{Another construction of canonical bases for quantum affine $\fks\fkl_n$}

Keep the setting as Section \ref{1}. With the PBW basis $\{E_A\mid A\in\ttz_\vartri^{ap}(n)\}$, the canonical basis of $U_\cz^+$ is constructed by elementary linear algebra method.

The bar involution \cite[Prop 7.5]{VaragnoloVasserot1999decomposition} given by $^-:\fkH_\vartri(n)\ra \fkH_\vartri(n),m^{(A)}\mapsto m^{(A)}$ and $v\mapsto v^{-1}$ can be restricted to $U_\cz^+$. By restricting to $\ttz_\vartri^{ap}(n)_{\bf d},~{\bf d}\in\bbn^n_\vartri$, since
$$m^{(A)}=E_A+\sum_{B\in\ttz_\vartri^{ap}(n)_{\bf d},B\prec A}f_{B,A}E_B,~\text{where}~f_{B,A}=v^{\dz(A)-\dz(B)}\gz^B_{A}(v^2),$$
solving $E_A$ as

$$E_A=m^{(A)}+\sum_{B\in\ttz_\vartri^{ap}(n)_{\bf d},B\prec A}g_{B,A}m^{(B)}.$$

Applying bar involution $^-$, we obtain
$$\ol{E_A}=m^{(A)}+\sum_{B\in\ttz_\vartri^{ap}(n)_{\bf d},B\prec A}\ol{g_{B,A}}m^{(B)}=E_A+\sum_{B\in\ttz_\vartri^{ap}(n)_{\bf d},B\prec A}r_{B,A}E_B$$

By \cite[7.10]{Lusztig1990canonical} (or \cite[\S0.5]{DengDuParashallWang2008finite},\cite{Du1994ic}), the system
$$p_{B,A}=\sum_{B\preccurlyeq C\preccurlyeq A}r_{B,C}\ol{p_{C,A}}\quad \text{for}~B\preccurlyeq A, A,B\in\ttz_\vartri^{ap}(n)_{\bf d}$$
has a unique solution satisfying $p_{A,A}=1,p_{B,A}\in v^{-1}\bbz[v^{-1}]$ for $B\prec A$. Moreover, the elements
$$\mathbf{C}_A=\sum_{B\preccurlyeq A,B\in \ttz_\vartri^{ap}(n)}p_{B,A}E_B,\quad A\in\ttz_\vartri^{ap}(n)_{\bf d}$$
is a $\cz$-basis of $U_{\bf d}^+$. Thus, $$\scc(\wih{\fks\fkl}_n)=\{\mathbf{C}_A\mid A\in\ttz_\vartri^{ap}(n)\}$$ is
the canonical basis of $U_\cz^+$ with respect to the PBW basis $\{E_A\}_{A\in\ttz_\vartri^{ap}(n)}$.

Using the triangular relation in \eqref{triangluarrelation}, the canonical basis $\{\mathbf{c}_A\mid A\in\ttz_\vartri^+(n)\}$ (with respect to the defining basis $\wit{u}_A$) of $\fkH_\vartri(n)$ can be constructed by the same way, for detail, see \cite{DuZhaomultiplication}.

\begin{thm}[{\cite[Thm 9.2]{DengDuXiao2007generic}}]\label{compareofcanobasis}
  $\mathbf{c}_A$ is the same as the canonical basis defined by geometric method by Lusztig in \cite{Lusztig1992affine}. In particular, when restricted to $\ttz_\vartri^{ap}(n)$, we have $\{\mathbf{c}_A\mid A\in \ttz_{\vartri}^{ap}(n)\}=\scc(\wih{\fks\fkl}_n)$.
\end{thm}

Next, we review the construction \cite{DuZhaomultiplication} of canonical basis of $\fkH_\vartri(n)$ from $m^{(A)}$.

For $A\in\ttz_\vartri^{+}(n)$, define
$$\ci_{\prec A}=\{B\in\ttz_{\vartri}^+(n)\mid B\prec A\}=\ci_{\prec A}^{1}\cup \ci_{\prec A}^{2}\cup\cdots \ci_{\prec A}^{t}~\text{for some}~t\in\bbn,$$

where $\ci_{\prec A}^{1}=\{\text{maximal elements of}~\ci_{\prec A}\}$ and
$\ci_{\prec A}^{i}=\{\text{maximal elements of}~\ci_{\prec A}\setminus \cup_{j=1}^{i-1}\ci_{\prec A}^j\}$ for $2\leqs i\leqs t$.

For a fixed $A\in\ttz_{\vartri}^{+}(n)_{\bf d},~{\bf d}\in\bbn^n_\vartri$, since
$$m^{(A)}=\wit{u}_A+\sum_{B\prec A,B\in\ttz_\vartri^+(n)_{\bf d}}h_{B,A}\wit{u}_B,~h_{B,A}\in\bbz[v,v^{-1}],$$
choose $B\in\ttz_{\vartri}^{+}(n)_{\bf d}$ such that $h_{B,A}\notin v^{-1}\bbz[v^{-1}]$ with $B\in \ci_{\prec A}^a$ and $a$ is minimal, then
$h_{B',A}\in v^{-1}\bbz[v^{-1}]$ if $B\prec B'$, or $B'\in \ci_{\prec A}^i,~i<a$.

Note the fact that every $h_{B,A}$ in $\bbz[v,v^{-1}]$ has a unique decomposition $h_{B,A}=h'_{B,A}+h''_{B,A}$
with $\ol{h'_{B,A}}=h'_{B,A}$ and $h''_{B,A}\in v^{-1}\bbz[v^{-1}]$.

\begin{lem}[{\cite[Algo 5.5]{DuZhaomultiplication}}]\label{another const for hall alg}
For $A\in\ttz_{\vartri}^{+}(n)_{\bf d},~{\bf d}\in\bbn^n_\vartri$, there exists $h'_{B,A}\in\bbz[v,v^{-1}]$ with $\ol{h'_{B,A}}=h'_{B,A}$ such that
  $$m^{(A)}-\sum_{\stackrel{B\in\ci_{\prec A}}{B\in\ttz_{\vartri}^{+}(n)}}h'_{B,A}m^{(B)}$$
  is the canonical basis of $\fkH_\vartri(n)$ via the defining basis $\{\wit{u}_A\}_{A\in\ttz_{\vartri}^{+}(n)}$.
\end{lem}

If $A$ is restricted to $\ttz_\vartri^{ap}(n)$, we have the following similar result for $U_\cz^+$ directly from $m^{(A)}$, which is very convenient to deal with canonical basis from $\wit{u}_{A}$ instead of $E_A$.

\begin{lem}\label{construction of cano basis for affine sln}
For $A\in\ttz_{\vartri}^{ap}(n)_{\bf d},~{\bf d}\in\bbn^n_\vartri$, there exists $h'_{B,A}\in\bbz[v,v^{-1}]$ with $\ol{h'_{B,A}}=h'_{B,A}$ such that
  $$m^{(A)}-\sum_{\stackrel{B\in\ci_{\prec A}}{B\in\ttz_{\vartri}^{ap}(n)}}h'_{B,A}m^{(B)}=\wit{u}_A
  +\sum_{\stackrel{B\in\ci_{\prec A}}{B\in\ttz_{\vartri}^{ap}(n)}}v^{-1}\bbz[v^{-1}]\wit{u}_B+\sum_{\stackrel{C\in\ci_{\prec A}}{C\in\ttz_{\vartri}^{p}(n)}}\bbz[v,v^{-1}]\wit{u}_C,$$
  then this is the canonical basis of $U_\cz^+$ via the PBW basis $\{E_A\}_{A\in\ttz_\vartri^{ap}(n)}$, and more precisely, the coefficient of $\wit{u}_C$ for those $C\in\ttz_\vartri^{p}(n)$ exactly belongs to $v^{-1}\bbz[v^{-1}]$.
\end{lem}

\begin{proof}
Since $$m^{(A)}=\wit{u}_A+\sum_{\stackrel{B\in\ci_{\prec A}}{B\in\ttz_{\vartri}^{ap}(n)}}h_{B,A}\wit{u}_B
+\sum_{\stackrel{C\in\ci_{\prec A}}{C\in\ttz_{\vartri}^{p}(n)}}\bbz[v,v^{-1}]\wit{u}_C,$$

choose $B\in\ttz_{\vartri}^{ap}(n)_{\bf d}$ such that $h_{B,A}\notin v^{-1}\bbz[v^{-1}]$ with $B\in \ci_{\prec A}^a$ and $a$ is minimal, then
$h_{B',A}\in v^{-1}\bbz[v^{-1}]$ if $B\prec B'$, or $B'\in \ci_{\prec A}^i,~i<a$.

Then
$$m^{(A)}-\sum_{B\in\ci_{\prec A}^a}h'_{B,A}m^{(B)}=\wit{u}_{A}+\sum_{\stackrel{B\in\ci_{\prec A}^i}{i<a}}v^{-1}\bbz[v^{-1}]
\wit{u}_B+\sum_{\stackrel{B\in\ci_{\prec A}^i}{i>a}}g_{B,A}\wit{u}_B+\sum_{\stackrel{C\in\ci_{\prec A}}{C\in\ttz_{\vartri}^{p}(n)}}\bbz[v,v^{-1}]\wit{u}_C.$$

By a similar argument with $g_{B,A},B\in\ttz_{\vartri}^{ap}(n)$ and continuing if necessary, we eventually obtain
$$m^{(A)}-\sum_{\stackrel{B\in\ci_{\prec A}}{B\in\ttz_{\vartri}^{ap}(n)}}h'_{B,A}m^{(B)}=\wit{u}_A
  +\sum_{\stackrel{B\in\ci_{\prec A}}{B\in\ttz_{\vartri}^{ap}(n)}}v^{-1}\bbz[v^{-1}]\wit{u}_B+\sum_{\stackrel{C\in\ci_{\prec A}}{C\in\ttz_{\vartri}^{p}(n)}}\bbz[v,v^{-1}]\wit{u}_C.$$

Using the relation $E_A=\wit{u}_A+\sum_{\stackrel{C_1\in\ci_{\prec A}}{C_1\in\ttz_{\vartri}^{p}(n)}}v^{-1}\bbz[v^{-1}]\wit{u}_{C_1}$, we obtain
\begin{equation*}
  \begin{split}
    &m^{(A)}-\sum_{\stackrel{B\in\ci_{\prec A}}{B\in\ttz_{\vartri}^{ap}(n)}}h'_{B,A}m^{(B)}\\
    =&(E_A-\sum_{\stackrel{C_1\in\ci_{\prec A}}{C_1\in\ttz_{\vartri}^{p}(n)}}v^{-1}\bbz[v^{-1}]\wit{u}_{C_1})+\sum_{\stackrel{B\in\ci_{\prec A}}{B\in\ttz_{\vartri}^{ap}(n)}}v^{-1}\bbz[v^{-1}](E_B-\sum_{\stackrel{C_2\in\ci_{\prec B}}{C_2\in\ttz_{\vartri}^{p}(n)}}v^{-1}\bbz[v^{-1}]\wit{u}_{C_2})\\
    &+\sum_{\stackrel{C\in\ci_{\prec A}}{C\in\ttz_{\vartri}^{p}(n)}}\bbz[v,v^{-1}]\wit{u}_C\\
    =&E_A+\sum_{\stackrel{B\in\ci_{\prec A}}{B\in\ttz_{\vartri}^{ap}(n)}}v^{-1}\bbz[v^{-1}]E_B\\
    &+\sum_{\stackrel{C\in\ci_{\prec A}}{C\in\ttz_{\vartri}^{p}(n)}}\bbz[v,v^{-1}]\wit{u}_C-\sum_{\stackrel{C_1\in\ci_{\prec A}}{C_1\in\ttz_{\vartri}^{p}(n)}}v^{-1}\bbz[v^{-1}]\wit{u}_{C_1}
    -\sum_{\stackrel{C_2\in\ci_{\prec A}}{C_2\in\ttz_{\vartri}^{p}(n)}}v^{-1}\bbz[v^{-1}]\wit{u}_{C_2}.
  \end{split}
\end{equation*}

By Lemma \ref{decomp hall alg as v.s}, we have

\begin{equation*}
\begin{split}
  &\sum_{\stackrel{C\in\ci_{\prec A}}{C\in\ttz_{\vartri}^{p}(n)}}\bbz[v,v^{-1}]\wit{u}_C-\sum_{\stackrel{C_1\in\ci_{\prec A}}{C_1\in\ttz_{\vartri}^{p}(n)}}v^{-1}\bbz[v^{-1}]\wit{u}_{C_1}
    -\sum_{\stackrel{C_2\in\ci_{\prec A}}{C_2\in\ttz_{\vartri}^{p}(n)}}v^{-1}\bbz[v^{-1}]\wit{u}_{C_2}\\
    &=m^{(A)}-\sum_{\stackrel{B\in\ci_{\prec A}}{B\in\ttz_{\vartri}^{ap}(n)}}h'_{B,A}m^{(B)}-(E_A+\sum_{\stackrel{B\in\ci_{\prec A}}{B\in\ttz_{\vartri}^{ap}(n)}}v^{-1}\bbz[v^{-1}]E_B)\in \U^+\cap {\bf P}=0.
\end{split}
\end{equation*}

Thus,

\begin{equation*}
    m^{(A)}-\sum_{\stackrel{B\in\ci_{\prec A}}{B\in\ttz_{\vartri}^{ap}(n)}}h'_{B,A}m^{(B)}=E_A+\sum_{\stackrel{B\in\ci_{\prec A}}{B\in\ttz_{\vartri}^{ap}(n)}}v^{-1}\bbz[v^{-1}]E_B.
\end{equation*}

Note that $\ol{m^{(A)}-\sum_{\stackrel{B\in\ci_{\prec A}}{B\in\ttz_{\vartri}^{ap}(n)}}h'_{B,A}m^{(B)}}=m^{(A)}-\sum_{\stackrel{B\in\ci_{\prec A}}{B\in\ttz_{\vartri}^{ap}(n)}}h'_{B,A}m^{(B)}$, by uniqueness of the canonical basis of $U_\cz^+$
with respect to the PBW basis $\{E_A\}_{A\in\ttz_\vartri^{ap}(n)}$, the result follows.
\end{proof}

Before ending this section, we give an example.
\begin{exam}Keep the setting as in example \ref{example 1}, consider $$\Pi_{\bf d}^{\preccurlyeq A_9}=\{A_9,A_{13},A_6,A_{12},A_{16},A_{15},A_{18}\}.$$

The Hasse diagram of $(\Pi_{\bf d}^{\preccurlyeq A_9},\preccurlyeq)$ has the form
$$
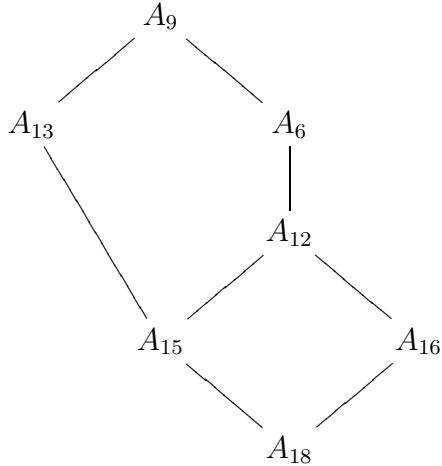
\xymatrix{
 && A_9\ar@{-}[dl]\ar@{-}[dr]& & &&\\
&A_{13}\ar@{-}[ddr]&& A_6\ar@{-}[d]  &  &  \\
&& &A_{12}\ar@{-}[dl]\ar@{-}[dr]&& &&\\
 & &A_{15}\ar@{-}[dr]&& A_{16}\ar@{-}[dl] & \\
  &&& A_{18}
}$$

Using the matrix algorithm in \cite{DuZhaomultiplication}, we take the following distinguished words $12^23^3\in\wp^{-1}(A_9),2^213^3\in \wp^{-1}(A_{13}),123^32\in\wp^{-1}(A_6),
213^32\in\wp^{-1}(A_{12}),13^32^2\in\wp^{-1}(A_{16})$.
\begin{equation*}
  \begin{split}
    m^{(A_9)}&=\wit{u}_{A_9}+v^{-2}\wit{u}_{A_{13}}+v^{-2}\wit{u}_{A_6}+v^{-3}\wit{u}_{A_{12}}+v^{-6}\wit{u}_{A_{16}}+v^{-4}\wit{u}_{A_{15}}+v^{-8}\wit{u}_{A_{18}},\\
    m^{(A_{13})}&=\wit{u}_{A_{13}}+v^{-2}\wit{u}_{A_{15}}+v^{-6}\wit{u}_{A_{18}},\\
    m^{(A_{6})}&=\wit{u}_{A_{6}}+v^{-1}\wit{u}_{A_{12}}+v^{-2}\wit{u}_{A_{15}}+(v^{-2}+v^{-4})\wit{u}_{A_{16}}+(v^{-4}+v^{-6})\wit{u}_{A_{18}},\\
    m^{(A_{12})}&=\wit{u}_{A_{12}}+v^{-3}\wit{u}_{A_{16}}+v^{-1}\wit{u}_{A_{15}}+(v^{-3}+v^{-5})\wit{u}_{A_{18}},\\
    m^{(A_{16})}&=\wit{u}_{A_{16}}+v^{-2}\wit{u}_{A_{18}}.\\
  \end{split}
\end{equation*}

By Lemma \ref{construction of cano basis for affine sln}, they are just the canonical bases associated with $A_9,A_{13},A_6,A_{12},A_{16}$, respectively.
Here we omit the calculation of those PBW basis $E_A$.
\end{exam}

\section{Application to $\U_v^+(\wih{\fks\fkl}_2)$}
The slices of canonical bases of $\U_v^+(\wih{\fkg\fkl}_2)$ is defined in \cite{DuZhaomultiplication} to ease the difficult of computing the canonical bases, where the authors have determined all the canonical bases of $\U_v^+(\wih{\fkg\fkl}_2)$
associated to $A$ with $\ell(A)\leqs 2$. When restricted to those $A\in\ttz_\vartri^{ap}(2)$, the
canonical bases of $\U_v^+(\wih{\fks\fkl}_2)$ with $\ell(A)\leqs 2$ is obtain by Theorem \ref{thm 1} \& \ref{compareofcanobasis}. In this section, we will determine all the canonical bases of $\U_v^+(\wih{\fks\fkl}_2)$ associated to $A\in\ttz^{ap}_\vartri(2)$ with $\ell(A)=3$.

The slices of the canonical(resp. monomial) basis is defined according to the Loewy length. For $l\in\bbn_{>0}$, let

\begin{equation*}
  \scc(\wih{\fks\fkl}_n)_l=\{\mathbf{C}_A\mid \ell(A)=l\}\quad(\text{resp.}~\scm(\wih{\fks\fkl}_n)_l=\{m^{(A)}\mid \ell(A)=l\}),
\end{equation*}
which is called a {\it canonical (resp., monomial) slice}. Note that this is specialization of $\mathscr{C}(\wih{\fkg\fkl}_n)_{(l,p)}$ and $\mathscr{M}(\wih{\fkg\fkl}_n)_{(l,p)}$ with $p=0$ in \cite{DuZhaomultiplication}. Clearly, each of the bases is a disjoint union of slices.

For our purpose, we first modify the multiplication formula given in the middle of \cite[p. 42]{DuFu2015quantum}.

\begin{thm}[\cite{DuFu2015quantum}]\label{twisted-multip-theorem}
For $\az\in\bbn_\vartri^n,A\in\ttz_\vartri^+(n)$, the twisted multiplication formula in the Ringel-Hall algebra $\fkH_\vartri(n)$ over $\cz$ is given by:
\begin{equation*}
  \widetilde{u}_\az\widetilde{u}_A=\sum_{\stackrel{T\in \Theta_{\vartriangle}^+(n)}{\row(T)=\az}}v^{f'_{A,T}}\prod_{\stackrel{1\leqs i\leqs n}{j\in\bbz,j>i}}
  \begin{bmatrix} a_{ij}+t_{ij}-t_{i-1,j}\\ t_{ij} \end{bmatrix}\wit{u}_{A+T-\wih{T}^+},
\end{equation*}
where
\begin{equation*}
  f'_{A,T}=\sum_{\stackrel{1\leqs i\leqs n}{j>l\geqs i+1}}a_{i,j}t_{i,l}-\sum_{\stackrel{1\leqs i\leqs n}{j>l\geqs i+1}}a_{i+1,j}t_{i,l}
  -\sum_{\stackrel{1\leqs i\leqs n}{j>l\geqs i+1}}t_{i-1,j}t_{i,l}+\sum_{\stackrel{1\leqs i\leqs n}{j>l\geqs i+1}}t_{i,j}t_{i,l},
\end{equation*}
and~ $~\wih{\ }:\ttz_\vartri(n)\ra\ttz_\vartri(n), X=(x_{i,j})\mapsto \wih{X}=(\wih{x}_{i,j})$
is the row-descending map defined by $\wih{x}_{i,j}=x_{i-1,j}$ for all $i,j\in\bbz$ and $\wih{T}^+$ denotes the upper triangular submatrix of $\wih T$.

\end{thm}
\begin{proof}
  By the formula on \cite[p. 42]{DuFu2015quantum}, we have
  \begin{equation*}
  \widetilde{u}_\az\widetilde{u}_A=\sum_{\stackrel{T\in \Theta_{\vartriangle}^+(n)}{\row(T)=\az}}v^{f_{A,T}}\prod_{\stackrel{1\leqs i\leqs n}{j\in\bbz,j>i}}
  \ol{\left[\!\!\left[\begin{matrix} a_{ij}+t_{ij}-t_{i-1,j}\\ t_{ij} \end{matrix}\right]\!\!\right]}\wit{u}_{A+T-\wih{T}^+},
\end{equation*}
where
\begin{equation*}
  f_{A,T}=\sum_{\stackrel{1\leqs i\leqs n}{j\geqs l\geqs i+1}}a_{i,j}t_{i,l}-\sum_{\stackrel{1\leqs i\leqs n}{j>l\geqs i+1}}a_{i+1,j}t_{i,l}
  -\sum_{\stackrel{1\leqs i\leqs n}{j\geqs l\geqs i+1}}t_{i-1,j}t_{i,l}+\sum_{\stackrel{1\leqs i\leqs n}{j>l\geqs i+1}}t_{i,j}t_{i,l}.
\end{equation*}

Using the identity $$\begin{bmatrix}
     N\\t
   \end{bmatrix}=v^{-t(N-t)}\left[\!\!\left[\begin{matrix} N\\ t \end{matrix}\right]\!\!\right]
   =v^{t(N-t)}\ol{\left[\!\!\left[\begin{matrix} N\\ t \end{matrix}\right]\!\!\right]},$$
we see
\begin{equation*}
\begin{split}
  f'_{A,T}&=\sum_{\stackrel{1\leqs i\leqs n}{j\geqs l\geqs i+1}}a_{i,j}t_{i,l}-\sum_{\stackrel{1\leqs i\leqs n}{j>l\geqs i+1}}a_{i+1,j}t_{i,l}
  -\sum_{\stackrel{1\leqs i\leqs n}{j\geqs l\geqs i+1}}t_{i-1,j}t_{i,l}+\sum_{\stackrel{1\leqs i\leqs n}{j>l\geqs i+1}}t_{i,j}t_{i,l}-\sum_{\stackrel{1\leqs i\leqs n}{j>i}}t_{ij}(a_{ij}-t_{i-1,j})\\
  &=\sum_{\stackrel{1\leqs i\leqs n}{j>l\geqs i+1}}a_{i,j}t_{i,l}-\sum_{\stackrel{1\leqs i\leqs n}{j>l\geqs i+1}}a_{i+1,j}t_{i,l}
  -\sum_{\stackrel{1\leqs i\leqs n}{j>l\geqs i+1}}t_{i-1,j}t_{i,l}+\sum_{\stackrel{1\leqs i\leqs n}{j>l\geqs i+1}}t_{i,j}t_{i,l}.
\end{split}
\end{equation*}
as desired.
\end{proof}

Next, we review the structure of $m^{(A)}$ with $A\in\ttz_\vartri^{ap}(2)$.

A sequence $(a_1,a_2,\cdots,a_l)\in\bbn^l$ is called a \emph{pyramidic} if there exists $k,~1\leqs k\leqs l$, such that
\begin{equation*}
  a_1\leqs a_2\leqs \cdots\leqs a_k,~a_k\geqs a_{k+1}\geqs \cdots\geqs a_{l}.
\end{equation*}

In the following, we identify $U_\cz^+$ with the composition algebra under the
isomorphism $\fkC_\vartri(n)\iso U_\cz^+,~u_i^{(m)}\mapsto E_i^{(m)}$ as given in Theorem \ref{isomorphism theorem for composition algebra}.

\begin{lem}[\cite{DuZhaomultiplication}]\label{structure of aperiodic part of monomial basis}We have
$$\scm(\wih{\fks\fkl}_2)=\{E_{i}^{(a_1)}E_{i+1}^{(a_2)}E_{i}^{(a_3)}E_{i+1}^{(a_4)}\cdots E_{i'}^{(a_l)}\mid i,i+1\in\bbz_2,(a_1,a_2,\cdots,a_l)~\text{is pyramidic},\forall l\in\bbn\},$$
where $i'=i$ if $l$ is odd and $i'=i+1$ if $l$ is even.
\end{lem}

Easy to see that $$\scm(\wih{\fks\fkl}_2)_3=\{E_{i}^{(a)}E_{j}^{(b)}E_{i}^{(c)}
\mid (i,j,i)=(1,2,1)~\text{or}~(2,1,2),(a,b,c)~\text{is pyramidic},a,b,c\in\bbn_{>0}\}.$$

More precisely, pyramidic sequence $(a,b,c)$ is equivalent to $a\geqs b\geqs c$, or $c\geqs b\geqs a$, or $b\geqs\{a,c\}$.

For $\bm i=(i_1,i_2,\cdots,i_t)\in I^t$ and $\bm a=(a_1,a_2,\cdots,a_t)\in\bbn^t$, define the monomial
$$E_{\bm i}^{(\bm a)}=E_{i_1}^{(a_1)}\cdots E_{i_t}^{(a_t)}.$$

Recall $\scc(\wih{\fks\fkl}_n)$ is the set of canonical bases of $U_\cz^+=U_v^+(\wih{\fks\fkl}_n)$. By \cite{Lusztig1993tight}, a monomial
$E_{\bm i}^{(\bm a)}$ is called {\it tight} if it belongs to $\scc(\wih{\fks\fkl}_n)$.

The following result will be proved in the appendix via the quadratic
form developed by \cite{DengDu2010tight}.

\begin{lem}\label{iff condition of tightness for 3 terms}
  For $\bm i=\{(1,2,1),(2,1,2)\}$ and $\bm a=\{(a,b,c)\in\bbn^3_{>0}\}$, $E_{\bm i}^{(\bm a)}$ is
  tight monomial if and only if $2b\geqs a+c$.
\end{lem}
\begin{rmk}
 If one of $a,b,c$ is zero, the tightness of
$E_{\bm i}^{(\bm a)}$ has been studied in \cite[Prop 6.4]{DuZhaomultiplication}.
\end{rmk}

We need the following identity for symmetric Gaussian polynomials.
\begin{lem}[{\cite[Sec 3.1]{XicanonicalA31999}}]\label{technic-lem-of-gauss-poly}
Assume that $m\geqs k\geqs 0,\dz\in\bbn$. Then
  \begin{equation*}
    \sum_{i=0}^\dz(-1)^iv^{i(m-k)}\begin{bmatrix}
    k-1+i\\k-1
  \end{bmatrix}\begin{bmatrix}
    m\\ \dz-i
  \end{bmatrix}=v^{-k\dz}\begin{bmatrix}
    m-k\\\dz
  \end{bmatrix}.
  \end{equation*}
\end{lem}

The main result in this section is the following.

\begin{thm}
Suppose $A\in\ttz_\vartri^{ap}(2)$ with $\ell(A)=3$, then there are 8 different matrices as follows
\begin{equation*}
\begin{split}
  A_1=
\begin{pmatrix}
  \bf 0 & a & b & c\\
  0 & \bf 0 & 0 & 0
\end{pmatrix},~A_2=\begin{pmatrix}
  \bf 0 & 0 & 0 & 0 & 0 \\
  0 & \bf 0 & a & b & c
\end{pmatrix},~A_3=\begin{pmatrix}
  \bf 0 & a & 0 & c \\
  0 & \bf 0 & 0 & b
\end{pmatrix},~A_4=\begin{pmatrix}
  \bf 0 & 0 & b & 0 & 0\\
  0 & \bf 0 & a & 0 & c
\end{pmatrix},\\
  A_5=
\begin{pmatrix}
  \bf 0 & 0 & b & c\\
  0 & \bf 0 & a & 0
\end{pmatrix},~A_6=\begin{pmatrix}
  \bf 0 & a & 0 & 0 & 0 \\
  0 & \bf 0 & 0 & b & c
\end{pmatrix},~A_7=\begin{pmatrix}
  \bf 0 & 0 & 0 & c \\
  0 & \bf 0 & a & b
\end{pmatrix},~A_8=\begin{pmatrix}
  \bf 0 & a & b & 0 & 0\\
  0 & \bf 0 & 0 & 0 & c
\end{pmatrix}.
\end{split}
\end{equation*}

Here $a,b,c\in\bbn,c>0$.

\begin{enumerate}[\rm(1)]
\item $m^{(A_5)},m^{(A_6)},m^{(A_7)},m^{(A_8)}$ are tight monomials.
\item For $i=1,2,3,4$, $m^{(A_i)}$ is tight monomial if and only if $a\leqs b$. If $a>b$, for $0\leqs k\leqs c$, set
\begin{equation*}
\begin{split}
 A_1^{(k)}&=
\begin{pmatrix}
  \bf 0 & a+c-k & b+c-k & k \\
  0 & \bf 0 & 0 & 0
\end{pmatrix},~\text{and}~m_1^{(k)}=m^{(A_1^{(k)})}, \\
 A_2^{(k)}&=
\begin{pmatrix}
  \bf 0 & 0 & 0 & 0 & 0 \\
  0 & \bf 0 & a+c-k & b+c-k & k
\end{pmatrix},~\text{and}~m_2^{(k)}=m^{(A_2^{(k)})},\\
 A_3^{(k)}&=
\begin{pmatrix}
  \bf 0 & a+c-k & 0 & k \\
  0 & \bf 0 & 0 & b+c-k
\end{pmatrix},~\text{and}~m_3^{(k)}=m^{(A_3^{(k)})},\\
 A_4^{(k)}&=
\begin{pmatrix}
  \bf 0 & 0 & b+c-k & 0 & 0\\
  0 & \bf 0 & a+c-k &0 & k
\end{pmatrix},~\text{and}~m_4^{(k)}=m^{(A_4^{(k)})},\\
\end{split}
\end{equation*}
(Note that $A_i^{(c)}=A_i,~m_i^{(c)}=m^{(A_i)}$)
then

\begin{equation*}
  \mathbf{C}_{A_{i}}=\sum_{k=0}^c(-1)^{c-k}\begin{bmatrix}
    a-b-1+c-k\\a-b-1
  \end{bmatrix}m^{(k)}_i
\end{equation*}
is the canonical basis associated to $A_i$. In other words,

\begin{equation*}
  \scc(\wih{\fks\fkl}_2)_3=\{\mathbf{C}_{A_{1}},\mathbf{C}_{A_{2}},\mathbf{C}_{A_{3}},\mathbf{C}_{A_{4}},m^{(A_5)},m^{(A_6)},m^{(A_7)},m^{(A_8)}\}.
\end{equation*}

\end{enumerate}
\end{thm}

\begin{proof}
  (1) By Lemma \ref{structure of aperiodic part of monomial basis}, we know $m^{(A_i)}(i=5,6,7,8)$ equals to
  \begin{equation*}
    E_1^{(b+c)}E_2^{(a+b+c)}E_1^{(c)},E_2^{(b+c)}E_1^{(a+b+c)}E_2^{(c)},E_1^{(c)}E_2^{(a+b+c)}E_1^{(b+c)},E_2^{(c)}E_1^{(a+b+c)}E_2^{(b+c)},
  \end{equation*}
which are tight by Lemma \ref{iff condition of tightness for 3 terms}.

  (2) We only deal with $A_1$, the others can be proved similarly. By Lemma \ref{structure of aperiodic part of monomial basis}, $m^{(A_1)}$ is
  corresponding to $E_1^{(a+b+c)}E_2^{(b+c)}E_1^{(c)}$, which is tight if and only if $a\leqs b$.

  From now on, suppose $a>b$. By Theorem \ref{twisted-multip-theorem}, for $0\leqs k\leqs c$, we have

  \begin{equation*}
    m_1^{(k)}=\wit{u}_{(a+b+2c-k)S_1}\wit{u}_{(b+c)S_2}\wit{u}_{kS_1}=\sum_{\stackrel{t_3\leqs t_1\leqs k}{t_2\leqs b+c-t_1}}
    v^{f}\begin{bmatrix}
     a+b+2c-t_1-t_2-t_3\\a+b+2c-k-t_2-t_3
   \end{bmatrix}\wit{u}_C,
  \end{equation*}
where $f=-(k-t_1)(b+c-t_1)+(a+b+2c-k-t_2-t_3)(t_2+t_3-b-c)+t_2(t_3-t_1)$ and

   \begin{equation*}
   C=\begin{pmatrix}
     \bf 0 & a+b+2c-t_1-t_2-t_3 & t_2 & t_3\\
     0 & \bf 0 & b+c-t_1-t_2 & t_1-t_3
   \end{pmatrix}.
 \end{equation*}

Then

\begin{equation*}
\begin{split}
  M(c)&=\sum_{k=0}^c(-1)^{c-k}\begin{bmatrix}
    a-b-1+c-k\\a-b-1
  \end{bmatrix}m^{(k)}_1\\
  &=\sum_{k=0}^c(-1)^{c-k}\begin{bmatrix}
    a-b-1+c-k\\a-b-1
  \end{bmatrix}\bigg(\sum_{\stackrel{t_3\leqs t_1\leqs k}{t_2\leqs b+c-t_1}}
    v^{f}\begin{bmatrix}
     a+b+2c-t_1-t_2-t_3\\a+b+2c-k-t_2-t_3
   \end{bmatrix}\wit{u}_C\bigg)\\
   &=\sum_{\substack{t_1=0\\t_3\leqs t_1\\t_2\leqs b+c-t_1}}^c\sum_{k=t_1}^c\bigg((-1)^{c-k}v^{f}\begin{bmatrix}
    a-b-1+c-k\\a-b-1
  \end{bmatrix}
\begin{bmatrix}
     a+b+2c-t_1-t_2-t_3\\a+b+2c-k-t_2-t_3
   \end{bmatrix}\bigg)\wit{u}_C\\
   &=\wit{u}_{A_1}+\sum_{\substack{t_1=0\\t_3\leqs t_1\\t_2\leqs b+c-t_1}}^{c-1}\bigg(\sum_{k=t_1}^c(-1)^{c-k}v^{f}\begin{bmatrix}
    a-b-1+c-k\\a-b-1
  \end{bmatrix}
\begin{bmatrix}
     a+b+2c-t_1-t_2-t_3\\a+b+2c-k-t_2-t_3
   \end{bmatrix}\bigg)\wit{u}_C.
\end{split}
\end{equation*}

By Lemma \ref{construction of cano basis for affine sln}, we only need to deal with those aperiodic matrices $C$. There are only four possible cases as follows

\begin{equation*}
  \begin{cases}
    a+b+2c=t_1+t_2+t_3\\
    t_2=0
  \end{cases},
  \begin{cases}
    a+b+2c=t_1+t_2+t_3\\
    t_1=t_3
  \end{cases},
  \begin{cases}
    b+c=t_1+t_2\\
    t_2=0
  \end{cases},
    \begin{cases}
    b+c=t_1+t_2\\
    t_1=t_3
  \end{cases}.
\end{equation*}

By definition, it is easy to see that the first two cases are impossible, and the third case appears unless $b=0$. Thus, under the fourth condition

\begin{equation*}
    \begin{cases}
    b+c=t_1+t_2\\
    t_1=t_3
  \end{cases}
\end{equation*}
and using the equivalent condition for degenerate order in \eqref{equivalent condition for degenerate order},
the aperiodic part of Hasse diagram of
$A_1=\begin{pmatrix}
  \bf 0 & a & b & c\\
  0 & \bf 0 & 0 & 0
\end{pmatrix}$ is the following

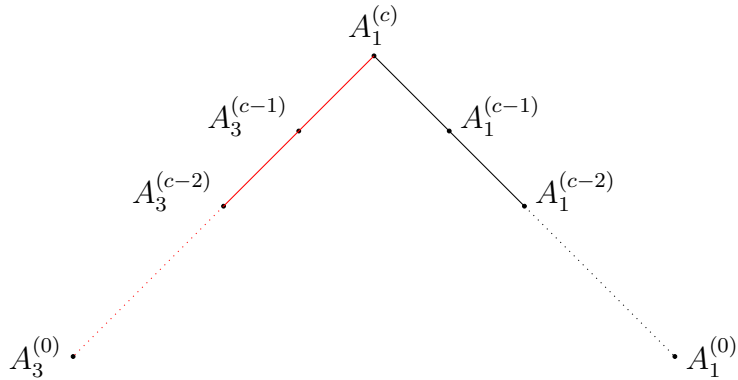
\begin{center}
\begin{tikzpicture}
\node [above] at (0,0) {$A_1^{(c)}$};
\node [above,right] at (1,-0.8) {$A_1^{(c-1)}$};
\node [above,left] at (-1,-0.8) {$A_3^{(c-1)}$};
\node [above,left] at (-2,-1.8) {$A_3^{(c-2)}$};\node [above,right] at (2,-1.8) {$A_1^{(c-2)}$};
\node [right] at (4,-4) {$A_1^{(0)}$};\node [left] at (-4,-4) {$A_3^{(0)}$};
\draw[fill] (0,0) circle [radius=0.025];\draw[fill] (1,-1) circle [radius=0.025];\draw[fill] (2,-2) circle [radius=0.025];
\draw[fill] (4,-4) circle [radius=0.025];
\draw[fill] (-1,-1) circle [radius=0.025];\draw[fill] (-2,-2) circle [radius=0.025];
\draw[fill] (-4,-4) circle [radius=0.025];
\draw (0,0)--(1,-1);\draw[red] (0,0)--(-1,-1);
\draw [red](-1,-1)--(-2,-2);\draw (1,-1)--(2,-2);
\draw [dotted,red](-2,-2)--(-4,-4);
\draw[dotted](2,-2)--(4,-4);
\end{tikzpicture}
\end{center}

Note that the red part appears unless $b=0$. If $b=0$ and \begin{equation*}
  \begin{cases}
    b+c=t_1+t_2\\
    t_2=0
   \end{cases},
\end{equation*}
then every coefficient of such $\wit{u}_{C}$ is
\begin{equation*}
\begin{split}
  &v^{(a+2c-k-t_3)(t_3-c)}\begin{bmatrix}
    a-1+c-k\\a-1
  \end{bmatrix}
\begin{bmatrix}
     a+c-t_3\\a+2c-k-t_3
   \end{bmatrix}~~(\text{for}~0\leqs k\leqs c,0\leqs t_3\leqs c-1)\\
  &=v^{-(a+2c-k-t_3)(c-t_3)-(c-k)(2c+1-k-t_3)}\ol{\left[\!\!\left[\begin{matrix} a-1+c-k\\ a-1\end{matrix}\right]\!\!\right]}
  \ol{\left[\!\!\left[\begin{matrix} a+c-t_3\\ a+2c-k-t_3\end{matrix}\right]\!\!\right]}\in v^{-1}\bbz[v^{-1}],
\end{split}
\end{equation*}
which means we only need to consider the coefficient of $\wit{u}_{A_1^{(c-1)}},\cdots,\wit{u}_{A_1^{(0)}}$.

Assume
\begin{equation*}
    \begin{cases}
    b+c=t_1+t_2\\
    t_1=t_3
  \end{cases},
\end{equation*}
then for $0\leqs t_1\leqs c-1$, the coefficient of $\wit{u}_{A_1^{(t_1)}}$ is

\begin{equation*}
\begin{split}
f(c,t_1)&=\sum_{k=t_1}^c(-1)^{c-k}v^{-(k-t_1)(b+c-t_1)}\begin{bmatrix}
    a-b-1+c-k\\a-b-1
  \end{bmatrix}
\begin{bmatrix}
     a+c-t_1\\a+c-k
   \end{bmatrix}\\
 &=\sum_{k'=0}^{c'}(-1)^{c'-k'}v^{-k'(b+c')}\begin{bmatrix}
    a-b-1+c'-k'\\a-b-1
  \end{bmatrix}
\begin{bmatrix}
     a+c'\\k'
   \end{bmatrix}~(k'=k-t_1,c'=c-t_1)\\
   &=v^{-c'(b+c')}\sum_{i=0}^{c'}(-1)^{i}v^{i(b+c')}\begin{bmatrix}
    a-b-1+i\\a-b-1
  \end{bmatrix}
\begin{bmatrix}
     a+c'\\c'-i
   \end{bmatrix}~(i=c'-k')
\end{split}
\end{equation*}

Let $k=a-b,m=a+c'$ and $\dz=c'$. Applying Lemma \ref{technic-lem-of-gauss-poly} gives

\begin{equation*}
f(c,t_1)=v^{-c'(b+c')}v^{-c'(a-b)}
   \begin{bmatrix}
     b+c'\\c'
   \end{bmatrix}=v^{-c'(a-b+c')}\ol{\left[\!\!\left[\begin{matrix} b+c'\\c'\end{matrix}\right]\!\!\right]}\in v^{-1}\bbz[v^{-1}],
\end{equation*}
since $a>b$. Hence, $M(c)=\wit{u}_{A_1}+\sum_{\stackrel{C_1\in\ci_{\prec A_1}}{C_1\in\ttz_{\vartri}^{ap}(2)}}v^{-1}\bbz[v^{-1}]\wit{u}_{C_1}+\sum_{\stackrel{C_2\in\ci_{\prec A_1}}{C_2\in\ttz_{\vartri}^{p}(2)}}\bbz[v,v^{-1}]\wit{u}_{C_2}$.
On the other hand, $\ol{M(c)}=M(c)$. Consequently, by Lemma \ref{construction of cano basis for affine sln}, $\mathbf{C}_{A_1}=M(c)$, as desired.
\end{proof}

\section{Relation to canonical bases of affine quantum Schur algebras}
Let $\fkS_{\vartri,r}$ be the group consisting of all permutations $w:\bbz\ra\bbz$ such that $w(i+r)=w(i)+r$ for $i\in\bbz$.
Define $s_i\in\fkS_{\vartri,r}$ by $s_i(j)=j$ for $j\not \equiv i,i+1(\mod r)$, $s_i(j)=j-1$ for $j\equiv i+1(\mod r)$
and $s_i(j)=j+1$ for $j\equiv i(\mod r)$. Let $\fkS_r=\lan s_i\mid 1\leqs i\leqs n\ran$ be the subgroup of $\fkS_{\vartri,r}$
and $\rho$ be the permutation of $\bbz$ sending $j$ to $j+1$ for all $j\in\bbz$. Note that each $w\in\fkS_{\vartri,r}$ can be
uniquely expressed as $w=\rho^ax$ with $a\in\bbz$ and $x\in \fkS_r$, and the length function $\ell$ on $\fkS_r$ can be extended to $\fkS_{\vartri,r}$
by setting $\ell(w)=\ell(x)$.

The extended affine Hecke algebra $\ch_\vartri(r)_\cz$ over $\cz$ associated to $\fkS_{\vartri,r}$ is the
unital $\cz$-algebra with basis $\{T_w\}_{w\in\fkS_{\vartri,r}}$, and multiplication defined by
\begin{equation*}
  \begin{cases}
    T_{s_i}^2=(v^2-1)T_{s_i}+v^2, &\text{for }1\leqs i\leqs r,\\
    T_wT_{w'}=T_{ww'}, &\text{if } \ell(ww')=\ell(w)+\ell(w').
  \end{cases}
\end{equation*}

Let $\bm\ch_\vartri(r)=\ch_\vartri(r)\otm_\cz\bbq(v)$. Denote by $\ch(\fkS_r)$ the $\cz$-subalgebra
of $\ch_{\vartri}(r)_\cz$ generated by $T_{s_i}$ for $1\leqs i\leqs r$. Let $\{C_w'\mid w\in \fkS_r\}$
be the canonical bases of $\ch(\fkS_r)$ defined in \cite{KazhdanLusztig1979repns}. For each $w\in \fkS_r$,
we have $C_w'=v^{-\ell(w)}\sum_{y\leqs w}P_{y,w}T_y$, where $P_{y,w}\in\bbz[v^2]$ is the Kazhidan-Lusztig polynomial.

For $\lz=(\lz_i)_{i\in\bbz}\in\bbz_{\vartri}^n,~A\in M_{\vartri,n}(\bbz)$, let
\begin{equation*}
  \sz(\lz)=\sum_{1\leqs i\leqs n}\lz_i\quad\text{and}\quad
  \sz(A)=\sum_{\substack{1\leqs i\leqs n\\j\in\bbz}}a_{i,j}=\sum_{\substack{1\leqs j\leqs n\\i\in\bbz}}a_{i,j}.
\end{equation*}

For $r\geqs 0$, we set
\begin{equation*}
  \llz_\vartri(n,r)=\{\lz\in\bbn_\vartri^n\mid \sz(\lz)=r\}~\text{and}~\ttz_\vartri(n,r)=\{A\in\ttz_\vartri(n)\mid \sz(A)=r\}.
\end{equation*}

For $\lz\in\llz_\vartri(n,r)$, let $\fkS_\lz=\fkS_{(\lz_1,\cdots,\lz_n)}$ be the corresponding standard Young subgroups of $\fkS_r$
and denote by $w_{0,\lz}$ the longest element in $\fkS_\lz$.
For each $\lz\in\llz_\vartri(n,r)$, let $x_\lz=\sum_{w\in\fkS_\lz}T_w\in\ch_\vartri(r)_\cz$. The endomorphism algebras
\begin{equation*}
  \cs_\vartri(n,r)_\cz=\End_{\ch_\vartri(r)_\cz}\bigg(\bps_{\lz\in\llz_\vartri(n,r)}x_\lz\ch_\vartri(r)_\cz\bigg)~\text{and}~
  \bm\cs_\vartri(n,r)_\cz=\End_{\bm\ch_\vartri(r)_\cz}\bigg(\bps_{\lz\in\llz_\vartri(n,r)}x_\lz\bm\ch_\vartri(r)\bigg)
\end{equation*}
are called {\em affine quantum Schur algebras}.(For details, cf. \cite{GinzburgVasserot1993langlands,Green1999affine,Lusztig1999aperiodicity}).

For $\lz\in\llz_\vartri(n,r)$, let $\scd_\lz^\vartri=\{d\mid d\in\fkS_{\vartri,r},\ell(wd)=\ell(w)+\ell(d)~\text{for}~w\in\fkS_\lz\}$ and
$\scd_{\lz,\mu}^\vartri=\scd_\lz^\vartri\cap{(\scd_{\mu}^\vartri)}^{-1}$. By \cite[7.4]{VaragnoloVasserot1999decomposition}(see also \cite[Lem 9.2]{DuFu2010modified}), there is a bijective map
\begin{equation*}
  \jmath_\vartri:\{(\lz,d,\mu)\mid d\in\scd_{\lz,\mu}^\vartri,\lz,\mu\in\llz_\vartri(n,r)\}\lra \ttz_\vartri(n,r),(\lz,d,\mu)\mapsto A=(|R_k^\lz\cap dR_l^\mu|)_{k,l\in\bbz},
\end{equation*}
where $$R_{i+kn}^\nu=\{\nu_{k,i-1}+1,\nu_{k,i-1}+2,\cdots,\nu_{k,i-1}+\nu_i=\nu_{k,i}\}~\text{with}~\nu_{k,i-1}=kr+\sum_{1\leqs t\leqs i-1}\nu_t,$$
for all $1\leqs i\leqs n,k\in\bbz$ and $\nu\in\llz_\vartri(n,r)$.

For $\lz,\mu\in\llz_\vartri(n,r)$ and $d\in\scd_{\lz,\mu}^\vartri$ satisfying $A=\jmath_\vartri(\lz,d,\mu)\in\ttz_\vartri(n,r)$,
define $e_A\in \cs_\vartri(n,r)_\cz$ by
\begin{equation*}
  e_A(x_\nu h)=\dz_{\mu\nu}\sum_{w\in\fkS_\lz d\fkS_\mu}T_wh,
\end{equation*}
where $\nu\in\llz_\vartri(n,r)$ and $h\in\ch_\vartri(n,r)_\cz$, and let
\begin{equation*}
  [A]=v^{-d_A}e_A,\quad\text{where}\quad d_A=\sum_{\substack{1\leqs i\leqs n\\i\geqs k,j<l}}a_{i,j}a_{k,l}.
\end{equation*}

For the geometrical description of $d_A$, see \cite{Lusztig1999aperiodicity}.
Note that the sets $\{e_A\mid A\in\ttz_\vartri(n,r)\}$ and $\{[A]\mid A\in\ttz_\vartri(n,r)\}$ form $\cz$-bases for $\cs_\vartri(n,r)_\cz$.

By definition, for $\lz\in\llz_\vartri(n,r)$ and $A\in\ttz_{\vartri}(n,r)$, we have
\begin{equation}\label{special equality}
  [A][\diag({\lz})]=
  \begin{cases}
    [A],&\text{if}~\lz=\col(A);\\
    0,&\text{otherwise}.
  \end{cases}\quad
  [\diag({\lz})][A]=
  \begin{cases}
    [A],&\text{if}~\lz=\row(A);\\
    0,&\text{otherwise}.
  \end{cases}
\end{equation}

Define the bar involution on $\cs_\vartri(n,r)_\cz$ via the one on the Hecke algebra. For $w=\rho^ax\in\fkS_{\vartri,r}$ with $a\in\bbz$ and $x\in \fkS_r$, let $C_w'=T_\rho^aC_x'$. Then $\{C_w'\mid x\in\fkS_{\vartri,r}\}$ forms a $\cz$-basis for $\ch_\vartri(r)_\cz$. Define bar involution on $\ch_\vartri(r)_\cz$ by $\bar{}:\ch_{\vartri}(r)_\cz\ra\ch_{\vartri}(r)_\cz,v\mapsto v^{-1},T_w\mapsto T_{w^{-1}}^{-1}$.
Following \cite{Du1992kahzdan} (or cf. \cite{VaragnoloVasserot1999decomposition}), the bar involution on $\cs_\vartri(n,r)_\cz$ can be described as follows:
\begin{equation}\label{bar-involution on affine q-schur alg}
  ^-:\cs_\vartri(n,r)_\cz\ra \cs_\vartri(n,r)_\cz,f\mapsto \bar{f},f\in\Hom_{\ch_\vartri(r)_\cz}(x_\mu\ch_\vartri(r)_\cz,x_\lz\ch_\vartri(r)_\cz),~h\in\ch_{\vartri}(r)_\cz,
\end{equation}
where $\bar{f}$ sends $v$ to $v^{-1}$ and $C_{w_0,\mu}'h$ to $\ol{f(C'_{w_{0,\mu}})}h$.

Let $$\ttz_\vartri^{\pm}(n)=\{A\in\ttz_\vartri(n)\mid a_{ii}=0~\text{for all}~i\}.$$

For $A\in\ttz_\vartri^\pm(n)$ and ${\bf j}\in\bbz_\vartri^n$, define elements in $\bm\cs_\vartri(n,r)$
\begin{equation*}
  A({\bf j},r)=\sum_{\mu\in\llz_\vartri(n,r-\sz(A))}v^{\mu\centerdot{\bf j}}[A+\diag(\mu)],
\end{equation*}
where $\mu\centerdot{\bf j}=\sum_{1\leqs i\leqs n}\mu_i{\bf j}_i$.
The set $\{A({\bf j},r)\}_{A\in\ttz_\vartri^\pm(n),{\bf j}\in\bbz_\vartri^n}$ spans $\bm\cs_\vartri(n,r)$.

Now we establish a triangular decomposition for $\cs_\vartri(n,r)_\cz$. Set the following $\cz$-submodules of $\cs_\vartri(n,r)_\cz$
\begin{equation*}
  \begin{split}
    \cs^+_\vartri(n,r)_\cz&={\rm span}_{\cz}\{A(0,r)\mid A\in\ttz_\vartri^+(n)\},\\
    \cs^-_\vartri(n,r)_\cz&={\rm span}_{\cz}\{A(0,r)\mid A\in\ttz_\vartri^-(n)\},\\
    \cs^0_\vartri(n,r)_\cz&={\rm span}_{\cz}\{[\diag(\lz)]\mid \lz\in\llz_\vartri(n,r)\}
  \end{split}
\end{equation*}
and $\ttz_\vartri^\pm(n)_{\leqs r}=\{A\in\ttz_\vartri^\pm(n)\mid \sz(A)\leqs r\}$.

\begin{thm}[{\cite[Prop 3.7.4~\& 3.7.7]{DengDuFu2012double}}]The elements $A(0,r)$(resp. $^tA(0,r)$) for $A\in\ttz_\vartri^+(n)_{\leqs r}$, form a $\cz$-basis of the subalgebra $\cs^+_\vartri(n,r)_\cz$(resp. $\cs^-_\vartri(n,r)_\cz$)
and $\{[\diag(\lz)]\}_{\lz\in\llz_{\vartri}(n,r)}$ form a $\cz$-basis
of the subalgebra $\cs^0_\vartri(n,r)_\cz$. In particular, we obtain a (weak) triangular decomposition:
\begin{equation*}
  \cs_\vartri(n,r)_\cz=\cs^+_\vartri(n,r)_\cz\cs^0_\vartri(n,r)_\cz\cs^-_\vartri(n,r)_\cz.
\end{equation*}
\end{thm}

Recall ${\bm\fkH}^-_\vartri(n)$ is the opposite algebra of ${\bm\fkH}^+_\vartri(n)(={\bm\fkH}_\vartri(n))$, and sometimes for sake of clarity,
write $u^+_A(=u_A)\in {\bm\fkH}^+_\vartri(n)$ and $u^-_A\in {\bm\fkH}^-_\vartri(n)$ and for $A\in\ttz_\vartri^+(n)$,
$$\wit{u}_A^\pm=v^{\dz(A)}u_{A}^\pm=v^{\dim\End(M(A))-\dim M(A)}u_{A}^\pm.$$

The relationship between $\fkD_\vartri(n)$ and $\bm\cs_\vartri(n,r)$ can be seen from the following
(cf. \cite{GinzburgVasserot1993langlands,Lusztig1999aperiodicity,VaragnoloVasserot1999decomposition}).
See also \cite[Thm 3.6.3 \& 3.81]{DengDuFu2012double}.

\begin{thm}
  For $r\in \bbn$, the map $\zeta_r:\fkD_\vartri(n)\ra\bm\cs_\vartri(n,r)$ is a surjective algebra homomorphism such that for all ${\bf j}\in\bbz_\vartri^n$
  and $A\in\ttz_\vartri^+(n)$,
  $$\zeta_r(K^{\bf j})=0({\bf j},r),~\zeta_r(\wit{u}_A^+)=A({\bf 0},r)~\text{and}~\zeta_r(\wit{u}_A^-)=({}^tA)({\bf 0},r).$$
  Here ${}^tA$ is the transpose matrix of A. In particular, $\zeta_r(\fkH^\pm_\vartri(n))\twoheadrightarrow\cs^\pm_\vartri(n,r)_\cz$ is also surjective.
\end{thm}

Now we recall the construction of canonical basis for affine quantum Schur algebras.

For $A,B\in \ttz_{\vartri}(n)$, define
\begin{equation}\label{partial order}
  B\sqsubseteq A~\text{if and only if}~B\preccurlyeq A,\col(B)=\col(A)~\text{and}~\row(B)=\row(A).
\end{equation}

Put $B\sqsubset A$ if $B\sqsubseteq A$ and $B\neq A$.

\begin{prop}[\cite{Lusztig1999aperiodicity,DuFu2014integral}]
  There exist canonical bases $\{\ttz_{A,r}\mid A\in\ttz_{\vartri}(n,r)\}$ for $\cs_\vartri(n,r)_\cz$ such that
  $\ol{\ttz_{A,r}}=\ttz_{A,r}$ and
  $$\ttz_{A,r}=[A]+\sum_{\stackrel{B\in\ttz_{\vartri}(n,r)}{B\sqsubset A}}g_{B,A,r}[B]\in [A]+\sum_{\stackrel{B\in\ttz_{\vartri}(n,r)}{B\sqsubset A}}v^{-1}\bbz[v^{-1}][B].$$
\end{prop}

Note that Lai and Luo \cite{LL} gave a matrix algorithm for those $[A]$, which can be used to construct the canonical basis for $\cs_\vartri(n,r)_\cz$.

Recall in Theorem \ref{compareofcanobasis}, we have introduced the canonical basis $\{\mathbf{c}_A\}_{A\in\ttz_\vartri^+(n)}$ of $\fkH_\vartri^+(n)$.
\begin{lem}\label{part-relation-of-cano-basis-lemma}
  For any $r>0$, set $\fkc_A=\zeta_r(\mathbf{c}_A)$ for all $A\in\ttz_\vartri^+(n,r)$. Then $\{\fkc_A\}_{A\in\ttz_\vartri^+(n)_{\leqs r}}$ forms a $\cz$-basis
  for $\cs^+_\vartri(n,r)_\cz$ which satisfies the following properties:
  \begin{equation*}
    \ol{\fkc_A}=\fkc_A\quad\text{and}\quad \fkc_A-A(0,r)\in\sum_{B\prec A}v^{-1}\bbz[v^{-1}]B(0,r).
  \end{equation*}
  In other words, this is the canonical basis relative to the basis $\{A(0,r)\mid A\in\ttz_\vartri^+(n)_{\leqs r}\}$ and the
  restrictions of the bar involution \eqref{bar-involution on affine q-schur alg} of $\cs_\vartri(n,r)_\cz$ and $\preccurlyeq$. Moreover, we have
  \begin{equation*}\zeta_r(\mathbf{c}_A)=
    \begin{cases}
      \fkc_A,&\text{if}~A\in\ttz_\vartri^+(n)_{\leqs r},\\
      0,&\text{otherwise}.
    \end{cases}
  \end{equation*}
  A similar result holds for $\cs_\vartri^-(n,r)_\cz$.
\end{lem}
\begin{proof}
  By Lemma \ref{another const for hall alg}, $\mathbf{c}_A=\wit{u}_{A}+\sum_{B\prec A}v^{-1}\bbz[v^{-1}]\wit{u}_{B}$ for any $A\in\ttz_\vartri^+(n)_{\bf d},{\bf d}\in\bbn^n_\vartri$. Applying $\zeta_r$ on both sides, we obtain
    \begin{equation*}
    \ol{\fkc_A}=\fkc_A\quad\text{and}\quad \fkc_A-A(0,r)\in\sum_{B\prec A}v^{-1}\bbz[v^{-1}]B(0,r).
  \end{equation*}

  From the ingredients of canonical basis, it is the required canonical basis.
\end{proof}

For $A\in\ttz_\vartri(n)$, define the ``hook sums''
\begin{equation*}
  \fkh_i(A)=a_{i,i}+\sum_{i<j}(a_{i,j}+a_{j,i})\quad\text{and}\quad \bm\fkh(A)=(\fkh_1(A),\cdots,\fkh_n(A))
\end{equation*}

If we write $A=A^++A^0+A^-$, where $A^\pm\in\ttz_\vartri^\pm(n)$ and $A^0$ is diagonal, then
\begin{equation*}
  \bm\fkh(A)=(a_{1,1},\cdots,a_{n,n})+\col(A^-)+\row(A^+).
\end{equation*}

Recall $\bm\fkh(A)\leqs \lz\llra \fkh_i\leqs \lz_i,\forall i.$

For $A\in\ttz_\vartri^+(n)$ and $\lz\in\llz(n,r)$ with $\bm\fkh(A)=\row(A)\leqs \lz$, let
$$A_\lz=A+\diag(\lz-\bm\fkh(A)).$$

By \eqref{special equality}, it is clear that $$A(0,r)[\diag(\lz)]=\sum_{\mu\in\llz_\vartri(n,r-\sz(A))}[A+\diag(\mu)][\diag(\lz)]=[A+\diag(\lz-\col(A))]$$
and $$[\diag(\lz)]A(0,r)=\sum_{\mu\in\llz_\vartri(n,r-\sz(A))}[\diag(\lz)][A+\diag(\mu)]=[A_\lz].$$

We have the following result similar to \cite[Thm 8.3]{DuParashall2002linear}.
\begin{thm}
  For $A\in\ttz_\vartri^+(n)$, if $\sz(A)\leqs r$, then $\fkc_A=\sum_{\bm\fkh(A)\leqs \lz}\ttz_{A_\lz,r}$. In other word, the image $\fkc_A$
  of the canonical basis elements $\mathbf{c}_A\in\fkH^+_\vartri(n)$ under $\zeta_r$ is either zero or a sum of the canonical basis elements $\ttz_{A_\lz,r}=[\diag(\lz)]\mathbf{c}_A\in\cs_\vartri(n,r)_\cz(\bm\fkh(A)\leqs r)$.
\end{thm}
\begin{proof}By Lemma \ref{part-relation-of-cano-basis-lemma} and \cite[Lem 7.2]{DuFu2014integral},
we have $\ol{\fkc_A}=\fkc_A$ and $\ol{[\diag{\lz}]}=[\diag(\lz)]$, thus, $\ol{[\diag(\lz)]\fkc_A}=[\diag(\lz)]\fkc_A$.

  By definition, $\fkc_A=A(0,r)+\sum_{B\prec A}t_{B,A}B(0,r)$, where $t_{B,A}\in v^{-1}\bbz[v^{-1}]$, it follows that
  \begin{equation*}
  \begin{split}
    [\diag(\lz)]\fkc_A&=[\diag(\lz)]A(0,r)+\sum_{B\prec A}t_{B,A}[\diag(\lz)]B(0,r)\\
    &=[A_\lz]+\sum_{B\prec A}t_{B,A}[B_{\lz}].
  \end{split}
  \end{equation*}

Easy to see, $B\prec A$ implies $B_\lz\prec A_{\lz}$. Also, $\row(A_\lz)=\row(B_\lz)=\lz$. Thus, if $\mu=\col(A_\lz)$,
$[\diag(\lz)]\fkc_A=\fkc_A[\diag(\mu)]$. Consequently, we obtain
  \begin{equation*}
  \begin{split}
    [\diag(\lz)]\fkc_A=[A_\lz]+\sum_{B\sqsubset A}t_{B,A}[B_{\lz}],
  \end{split}
  \end{equation*}
where $\sqsubset$ is the partial order relation on $\ttz_\vartri(n)$ defined in \eqref{partial order}.
Now, by the uniqueness of canonical basis relative to the basis $\{[A]\}_{A\in\ttz_\vartri(n)}$,
we must have $\ttz_{A_\lz,r}=[\diag(\lz)]\fkc_A$ and the result follows.
\end{proof}

\section{Appendix: Tightness of monomial basis via quadratic form}
In this appendix, we briefly review the criterion for tight monomials of quantum groups $\U^+$
associated to a quiver $(Q,\sz)$ with automorphism $\sz$ in \cite{DengDu2010tight} or \cite{Reineke2001monomials},
based on which we prove Lemma \ref{iff condition of tightness for 3 terms}.

Let $(Q,\sz)$ be a quiver $Q=(Q_0,Q_1)$ with automorphism $\sz$ over $k=\ol{\mathbb{F}}_q$ and assume $Q$ contains no oriented cycles.
Let $F_{Q,\sz,q}$ be the Frobenius morphism on the path algebra $A=kQ$ defined by
\begin{equation*}
  F=F_{Q.\sz,q}:A\ra A,~\sum_{s}x_sp_s\mapsto \sum_sx_s^q\sz(p_s),
\end{equation*}
where $\sum_{s}x_sp_s$ is a $k$-linear combination of path $p_s$. Thus, we have an $\mathbb{F}_q$-algebra
\begin{equation*}
  A^F=\{a\in A\mid F(a)=a\},
\end{equation*}
which is a finite dimensional hereditary $\mathbb{F}_q$-algebra.

Let $I$ be the set of isoclasses of simple modules in $A^F$-module. The Grothendieck group $K_0(A^F)$ of $A^F$ is then identified
with the free abelian group $\bbz I$ with basis $I$. Given a module $M$ in $A^F$-mod, we denote by $\bdim M$ the image of $M$ in
$K_0(A^F)$.

The Euler form $\lan-,-\ran:\bbz I\times\bbz I\lra\bbz$ associated with $(Q,\sz)$ is defined by
\begin{equation*}
  \lan {\bdim M,\bdim N}\ran=\dim_{\mathbb{F}_q}\Hom_{A^F}(M,N)-\dim_{\mathbb{F}_q} \Ext^1_{A^F}(M,N),
\end{equation*}
for $M,N\in A^F$-mod. Then the symmetric bilinear form on $\bbz I$ is defined by
\begin{equation*}
  \bm x\cdot \bm y=\lan \bm x,\bm y\ran+\lan \bm y,\bm x\ran,~\text{for}~\bm x,\bm y\in\bbz I.
\end{equation*}

Moreover, there exist unique Cartan datum $(I,\cdot)$ and generalized Cartan matrix $C_{Q,\sz}$ associated with $(Q,\sz)$(see \cite[1.1.1]{Lusztig1993introduction}),
and assume $\U^+$ is the quantum algebra associated with $(I,\cdot)$, which is generated by $E_i,i\in I$ and subjects to Serre relations.
By \cite{Lusztig1990canonical} or \cite{Kashiwara1991crystal}, there exist canonical bases for $\U^+$.

\begin{defn}
  For any fixed $\bm i=(i_1,i_2,\cdots,i_t)\in I^t$ and $\bm a=(a_1,a_2,\cdots,a_t)\in\bbn^t$, let $\scm_{\bm i,\bm a}$ be
  the set of $t\times t$ matrices $A=(a_{rm})$ with entries $a_{rm}$ in $\bbn$ satisfying the condition $\row(A)=\col(A)=\bm a$
  and $a_{rm}=0$ unless $i_r=i_m$. Define a quadratic form $\mathbbm{q}:\scm_{\bm i,\bm a}\ra \bbz$ by setting
  \begin{equation*}
    \mathbbm{q}(A)=\sum_{\substack{1\leqs m\leqs t\\1\leqs p<r\leqs t}}\lan i_m,i_m\ran a_{pm}a_{rm}
    +\sum_{\substack{1\leqs p<r\leqs t\\1\leqs l<m\leqs t}}(i_m\cdot i_m)a_{pm}a_{rl}
    +\sum_{\substack{1\leqs r\leqs t\\1\leqs l<m\leqs t}}\lan i_r,i_r\ran a_{rm}a_{rl},
  \end{equation*}
  for all $A\in \scm_{\bm i,\bm a}$.
\end{defn}

\begin{thm}[{\cite[Thm 2.5]{DengDu2010tight}}]
  Let $\U^+$ be the quantum algebra associated with a Cartan datum $(I,\cdot)$.
  For $\bm i=(i_1,i_2,\cdots,i_t)\in I^t$ and $\bm a=(a_1,a_2,\cdots,a_t)\in\bbn^t$, the monomial $E_{\bm i}^{(\bm a)}$ is
  tight(or a canonical basis element) if and only if $\mathbbm{q}(A)<0$ for all $A\in \scm_{\bm i,\bm a}\setminus\{D_{\bm a}\}$, where $D_{\bm a}=\diag(a_1,\cdots,a_t)$.
\end{thm}

From now on, suppose $(Q,\rm id)$ is the Kronecker quiver
$$Q:~\xymatrix{&1\ar@/^/[r]\ar@/_/[r]&2\\}$$
with identity automorphism, the associated Cartan datum is $(I,\cdot),~I=\{1,2\}$. Note that there exists an isomorphism between the twisted
composition subalgebras $\fkC_q(Q)\iso \fkC_\vartri(2)$, both of which give a realization of the positive part of
quantum algebra $\U_v(\wih{\fks\fkl}_2)$, for details, see \cite{Ringel1993composition,Ringel1993revisited}.

In the above setting, $\lan i,i\ran=1,i\cdot i=2$ for $i\in I$ and $\lan 1,2\ran=-2,\lan 2,1\ran=0,(1,2)=(2,1)=-2$.

\begin{proof}[Proof of Lemma \ref{iff condition of tightness for 3 terms}]
For $\bm i=\{(1,2,1),(2,1,2)\}$ and $\bm a=\{(a,b,c)\in\bbn_{>0}^3\}$, we have

\begin{equation*}
  \scm_{\bm i,\bm a}\setminus\{D_{\bm a}\}=\big\{A_x\mid 0<x\leqs \min\{a,c\}\big\},~A_x=\begin{bmatrix}
    a-x & 0 & x\\
    0 & b & 0\\
    x & 0 & c-x
  \end{bmatrix},
\end{equation*}
and
\begin{equation*}
\begin{split}
  \mathbbm{q}(A_x)&=\lan i_1,i_1\ran a_{11}a_{31}+\lan i_3,i_3\ran a_{31}a_{33}+(i_1\cdot i_2)a_{22}a_{31}+(i_2\cdot i_3)a_{13}a_{22}\\
  &\quad+(i_1\cdot i_3)a_{13}a_{31}+\lan i_1,i_1\ran a_{11}a_{13}+\lan i_3,i_3\ran a_{13}a_{33}\\
  &=(a-x)x+x(c-x)-2bx-2bx+2x^2+(a-x)x+(c-x)x\\
  &=-2x^2+2x(a+c-2b).
\end{split}
\end{equation*}

Easy to see, for $0<x\leqs \min\{a,c\}$, $\mathbbm{q}(A)<0\llra 2b\geqs a+c$. The result follows.
\end{proof}

\begin{cor}
For $\bm i=\{(1,2,1,2),(2,1,2,1)\}$ and $\bm a=\{(a,b,c,d)\in\bbn_{>0}^4\}$, $E_{\bm i}^{(\bm a)}$ is
tight monomial if and only if $2b\geqs a+c,~2c\geqs b+d$, but at least one is strict.
\end{cor}
\begin{proof}For $\bm i=\{(1,2,1,2),(2,1,2,1)\}$ and $\bm a=\{(a,b,c,d)\in\bbn_{>0}^4\}$, we have

\begin{equation*}
  \scm_{\bm i,\bm a}\setminus\{D_{\bm a}\}=\big\{A_{x,y}\mid 0\leqs x\leqs \min\{a,c\},0\leqs y\leqs \min\{b,d\},x^2+y^2\neq0\big\},
\end{equation*}

\begin{equation*}
  ~A_{x,y}=\begin{bmatrix}
    a-x & 0 & x & 0\\
    0 & b-y & 0 & y\\
    x & 0 & c-x & 0\\
    0 & y & 0 & d-y
  \end{bmatrix},
\end{equation*}
and
\begin{equation*}
\begin{split}
  \mathbbm{q}(A_{x,y})&=2\big(\lan i_1,i_1\ran x(a-x)+\lan i_2,i_2\ran y(b-y)+\lan i_3,i_3\ran x(c-x)+\lan i_4,i_4\ran y(d-y)\big)\\
  &\quad+\big((i_1\cdot i_2)+(i_2\cdot i_3)\big)x(b-y)+(i_1\cdot i_3)x^2+\big((i_1\cdot i_4)+(i_2\cdot i_3)\big)xy\\
  &\quad+(i_2\cdot i_4)y^2+\big((i_2\cdot i_3)+(i_3\cdot i_4)\big)y(c-x)\\
  &=-2(x-y)^2+2x(a+c-2b)+2y(b+d-2c).
\end{split}
\end{equation*}

If $2b\geqs a+c,~2c\geqs b+d$, and at least one is strict, then it is obvious that $\mathbbm{q}(A_{x,y})<0$ for
all $A_{x,y}\in \scm_{\bm i,\bm a}\setminus\{D_{\bm a}\}$.

If for all $A_{x,y}\in \scm_{\bm i,\bm a}\setminus\{D_{\bm a}\}$, $\mathbbm{q}(A_{x,y})<0$,
then take special $x=0,y\neq 0(x\neq 0,y=0,\text{and}~x=y\neq0)$, we obtain that
$2b\geqs a+c(2c\geqs b+d,\text{and}~b+c>a+d)$, which means $2b\geqs a+c,~2c\geqs b+d$, but at least one is strict. The result follows.
\end{proof}

\begin{rmk}
  The sufficient part was first given by Lusztig in \cite[Sec.12]{Lusztig1993tight}.
\end{rmk}


\end{document}